\documentclass[twoside,11pt]{amsart}
\usepackage{amsmath,latexsym,amssymb,times,enumerate}
\usepackage{amsthm}

\usepackage{verbatim}

\usepackage{enumitem}

\title{Generating Residual Intersections of Determinantal Ideals}
\author{Yevgeniya Tarasova}
\date{}

\usepackage{parskip}

\DeclareMathOperator{\depth}{depth}

\DeclareMathOperator{\height}{ht}

\DeclareMathOperator{\grade}{grade}

\DeclareMathOperator{\Ext}{Ext}
\DeclareMathOperator{\Hom}{Hom}

\DeclareMathOperator{\bideg}{bideg}
\DeclareMathOperator{\initial}{in}
\DeclareMathOperator{\Ass}{Ass}
\DeclareMathOperator{\support}{supp}

\DeclareMathOperator{\lcm}{lcm}

\usepackage{thmtools}
\usepackage{thm-restate}

\newtheorem{thm}{Theorem}[section]

\newtheorem{lem}[thm]{Lemma}
\newtheorem{cor}[thm]{Corollary}

\newtheorem{notation}[thm]{Notation}

\begin{document}

\maketitle

\begin{abstract}
   If $I$ is a perfect ideal in a local Cohen-Macaulay ring, the generators of ideals linked to $I$ are well understood. However, the generators of the residual intersections of $I$ have only been computed in a few special cases. In this paper, we show that the $n$-residual intersections of determinantal ideals of generic $2\times n$ matrices are sums of links. 
\end{abstract}

\section{Introduction}\label{Intro}

Let $(R,m)$ be a local Cohen-Macaulay ring and $I$ and $J$ be proper $R$-ideals. If there exists an ideal $\mathfrak{a}$ generated by a regular sequence such that $I = \mathfrak{a}:J$ and $J = \mathfrak{a}:I$, then $I$ and $J$ are said to be \textit{linked}. If, in addition, $\height(I+J) \geq \height(I)+1$, then we say $I$ and $J$ are \textit{geometrically linked} .

Residual intersections are a generalization of linkage. We say that $J$ is an \textit{$s$-residual intersection} of $I$ if there exists an ideal $\mathfrak{a} = (a_1,...,a_s) \subseteq I$ such that $J = \mathfrak{a}:I$ and $\height(J)\geq s$. The residual intersection is called \textit{geometric} if $\height(I+J) \geq s+1$. This notion of residual intersections was introduced by Artin and Nagata \cite{AN}.

Linkage has a long history and is well understood, particularly in connection to Cohen-Macaulay and Gorenstein properties. Notably, in 1974 Peskine and Szpiro \cite{PS} showed that if $R$ is a local Gorenstein ring and $I$ and $J$ are linked then $R/I$ is Cohen-Macaulay if and only if $R/J$ is Cohen-Macaulay. In the same paper Peskine and Szpiro \cite{PS} also showed that the canonical module of $R/I$ is $J/\mathfrak{a}$. Similarly, the canonical module of $R/J$ is $I/\mathfrak{a}$. It should also be noted that, if $I$ is a perfect ideal, then we can use a mapping cone construction to compute the generators of $J$.

Thus comes the natural question: Can we generalize these results to residual intersections? If $R$ is a local Gorenstein ring, $\mathfrak{a}\subset I$ and $J$ are $R$-ideals such that $J = \mathfrak{a}:I$ is an $s$-residual intersection and $R/I$ is Cohen-Macaulay, then $R/J$ is not necessarily Cohen-Macaulay. For example, if $R$ is $k[x_1,\dots,x_n,y_1,\dots,y_n]$ localized at its homogeneous maximal ideal and $n \geq 4$, the residual intersections of ideals generated by $2\times 2$ minors of generic $2 \times n$ matrices are not Cohen-Macaulay unless they are links.

However, under certain technical assumptions, it is possible to ensure that $R/J$ is Cohen-Macaulay. In a 1983 paper Huneke showed that if $I$ is strongly Cohen-Macaulay and satisfies the $G_s$ condition then $R/J$ is Cohen-Macaulay \cite{H}.  Later Herzog, Vasconcelos and Villarreal \cite{HVV} as well as Hunke and Ulrich \cite{HU} explored weakening the strongly Cohen-Macaulay assumption. Ulrich further generalized this work in 1994 and found settings where it is possible to compute the canonical module of $R/J$ \cite{U}. Research is still being done on settings where $R/J$ is Cohen-Macaulay and on the computation of the canonical module of $R/J$ \cite{HN, CNT}. Recently there has also been research on the relationship between residual intersections and the Gorenstein properties of Rees algebras \cite{EU}.

Unlike in the case of linkage, we only know the generators of residual intersections in a few special cases. If $R$ is Cohen-Macaulay, the generators are known in the following cases: When $I$ is a complete intersection \cite{HU}, when $I$ is a perfect ideal of height two \cite{H}, when $R/I$ is Gorenstein and $I$ is of height three \cite{KU}, and for certain $(\height(I)+1)$-residual intersections \cite{KMU}. Bouca and Hassanzadeh gave a formula to compute an $s$-residual intersection of $I$ when $I$ is an almost complete intersection by applying their more general results \cite{BH}. It has also been shown that one may express the residual intersections of ideals $I$ in terms of sums of links if the deviation of $I$ is less than or equal to 3, and some other technical assumptions are satisfied \cite{T}. 

The goal of this paper is to prove the following main result:

Let $\mathbf{k}$ be a field of characteristic 0. For an integer $n \geq 4$, let $R$ be $\mathbf{k}[x_1,\dots,x_n,y_1,\dots,y_n]$ localized at its homogeneous maximal ideal, let $M = \begin{pmatrix}
  x_1 \dots x_n \\
  y_1 \dots y_n
\end{pmatrix}$ be a $2\times n$ generic matrix, and let $I \subsetneq R$ be the ideal generated by the $2\times 2$ minors of $M$.

\begin{restatable*}{thm}{mainResult}
  \label{p2-main-result}%
  For any $n$-residual intersection $\mathfrak{a}:I$ of $I$, we can select a generating set $\{a_1,\dots,a_n\}$ of $\mathfrak{a}$ such that, for all $i$, $ (a_1,\dots,\widehat{a_i},\dots,a_n):I$ is a link and $\mathfrak{a}:I = \sum_{i=1}^n (a_1,\dots,\widehat{a_i},\dots,a_n):I$.
\end{restatable*}

In order to prove our main result, we first prove it in a special case. The special case is described in Section \ref{Sum of Links}. For this case, we select a particularly ``nice" $\mathfrak{a} = (a_1,\dots,a_n)$. The generators of $\mathfrak{a}$ are selected so that, for all $1 \leq i \leq n$,  $(a_1,\dots,\widehat{a_i},\dots,a_n):I$ is a geometric link, and we use graph theory to show that $\mathfrak{a}:I$ is a residual intersection (Lemma \ref{is Res Int}). Our choice of $\mathfrak{a}$ also provides an easier computation of the Gr\"obner basis of $\sum_{i=1}^n (a_1,\dots,\widehat{a_i},\dots,a_n):I$ when using reverse lexicographical order.

We use the results of Section \ref{Sec Links} to compute the generators of $\sum_{i=1}^n (a_1,\dots,\widehat{a_i},\dots,a_n):I$ (Lemma \ref{link is sum of monomials}). Then, after computing the Gr\"obner basis of $\sum_{i=1}^n (a_1,\dots,\widehat{a_i},\dots,a_n):I$ in Lemma \ref{Grobner basis of J}, we use the Gr\"obner Basis to prove that $\mathfrak{a}:I = \sum_{i=1}^n (a_1,\dots,\widehat{a_i},\dots,a_n):I$ (Theorem \ref{sum of links}).

This brings us to Section \ref{Main Section} wherein we prove our main result. In this section, we also prove Theorem \ref{general specialization}, which provides sufficient conditions for colon ideas to commute with surjective maps. This theorem, while interesting on its own, is also critical to proving our main result as it allows us to between our special case, generic residual intersections, and arbitrary residual intersections. 

\section{Links}\label{Sec Links}

The purpose of this section is to establish notation and compute the generators for a special case of links of determinantal ideals of $2 \times n$ matrices, so that we may use these results in Section \ref{Sum of Links}.

Let $\mathbf{k}$ be a field of characteristic 0 with $n \geq 4$, and let $R = \mathbf{k}[x_1,\dots,x_n,y_1,\dots,y_n]$. It should be noted that $R$ is naturally bigraded with $\bideg(x_i)=(1,0)$ and $\bideg(y_i)=(0,1)$ for all integers $i$.

Throughout this chapter we will be using the following notation: Let $[i,j]$ be the set of integers, $t$, such that $i \leq t \leq j$ where $i$ and $j$ are integers, and $i \leq j$. If $i > j$, then we set $[i,j] = \emptyset$. Let $\{i_1,\dots,i_k\} \subseteq [ 1,n]$ be a nonempty set. We define $X_{\{i_1,\dots,i_k\}}$ as the product $x_{i_1}\dots x_{i_k}$ and $Y_{\{i_1,\dots,i_k\}}$ as the product $y_{i_1}\dots y_{i_k}$. We define $X_\emptyset$ and $Y_\emptyset$ as 1. Notice that this definition implies that $X_{\{i_1,\dots,i_k\}}$ and $Y_{\{i_1,\dots,i_k\}}$ are square free monomials.

Let $M = \begin{pmatrix}
  x_1 \dots x_n \\
  y_1 \dots y_n
\end{pmatrix}$. When $i \neq j$, we denote the $2 \times 2$ minors of $M$ by $\Delta_{i,j} = x_iy_j-x_jy_i$. Note that when $i \neq j$, $\Delta_{i,j}$ is a bihomogeneous element of $R$. For notational convenience we set $\Delta_{i,i} = 0$. Let $I = I_2(M)$ be the ideal generated by the $2 \times 2$ minors of $M$, and let $\mathfrak{a} = (\Delta_{1,2}, \Delta_{2,3}, \dots, \Delta_{n-1,n})$. In this section, we will show that $\mathfrak{a}:I$ is a link (Lemma \ref{linkgens}) and compute its generators (Theorem \ref{generators of link}).

Note that the following lemma implies that $\mathfrak{a}:I$ is a link.

\begin{lem}
\label{linkgens}
  The sequence $\Delta_{1,2}, \Delta_{2,3}, \dots, \Delta_{n-1,n}$ is a regular sequence.
\end{lem}

\begin{proof}
  Let $\mathfrak{a} = (\Delta_{1,2}, \Delta_{2,3}, \dots, \Delta_{n-1,n})$. Define $\phi: R \rightarrow \mathbf{k}[x_1,\dots,x_n] $ as the following specialization map:
  \begin{equation*}
    \begin{pmatrix}
      x_1 \dots x_n \\
      y_1 \dots y_n
    \end{pmatrix} \rightarrow \begin{pmatrix}
      x_1 \dots x_{n-1}\: 0 \\
      0 \: x_1 \dots x_{n-1}
    \end{pmatrix}.
  \end{equation*}

Note that $\phi$ is surjective, that $\phi(\mathfrak{a}) = (x_1^2,x_2^2-x_1x_3,x_3^2 - x_2x_4, \dots, x_{n-2}^2-x_{n-3}x_{n-1}, x_{n-1}^2)$, and that $\sqrt{\phi(\mathfrak{a})} = (x_1, \dots, x_{n-1})$. Thus, $\height(\phi(\mathfrak{a})) = n-1$. Since $R$ is a Cohen-Macaulay ring, and since $\phi$ is surjective with a kernel generated by a regular sequence in $R$, $\height(\mathfrak{a}) \geq \height(\phi(\mathfrak{a})) = n-1$.  Furthermore, as $R$ is a Cohen-Macaulay ring, $\height(\mathfrak{a})=\grade(\mathfrak{a})$. Thus, $\Delta_{1,2}, \Delta_{2,3}, \dots, \Delta_{n-1,n}$ is a  regular sequence.\end{proof}

The following lemma will be used to help show that our candidate for the generating set of $\mathfrak{a}:I$ is contained in $\mathfrak{a}:I$.

\begin{lem}
\label{in a}
Fix integers $i$ and $j$ such that $1 \leq i < j \leq n$ and let $\mathfrak{a} = (\{\Delta_{t,t+1}\:\vert \: 1 \leq t \leq n-1 \})$. Let $K$ and $L$ be sets such that $K\cup L = [ i+1, j-1]$ and $K \cap J = \emptyset$. Then $X_{K}Y_{L}\Delta_{i,j} \in \mathfrak{a}$.
\end{lem}

\begin{proof}
Fix $i$. We will proceed by induction on $j$. Our induction begins with $j = i+1$ as $j>i$. By definition $\Delta_{i,i+1}\in \mathfrak{a}$.

Inductive step: Now, pick $j < n$ and suppose for $\{i_1,\dots,i_k\} \cup \{j_1,\dots,j_l\} = [ i+1, j-1 ]$ and $\{i_1,\dots,i_k\} \cap \{j_1,\dots,j_l\} = \emptyset$, $X_{ \{i_1,\dots,i_k\}}Y_{\{j_1,\dots,j_l\}}\Delta_{i,j} \in \mathfrak{a}$. Note that $\Delta_{j,j+1} \in \mathfrak{a}$.

We have the following equalities:
\begin{equation*}
  y_{j+1}\Delta_{i,j}+y_i\Delta_{j,j+1} = y_j \Delta_{i,j+1}
\end{equation*}
and
\begin{equation*}
  x_{j+1}\Delta_{i,j}+x_i\Delta_{j,j+1} = x_j \Delta_{i,j+1}.
\end{equation*}
Thus, we have:
\begin{equation*}
  y_{j+1}(X_{\{i_1,\dots,i_k\}}Y_{\{j_1,\dots,j_l\}}\Delta_{i,j}) + y_{i}X_{\{i_1,\dots,i_k\}}Y_{\{j_1,\dots,j_l\}}(\Delta_{j,j+1}) = X_{\{i_1,\dots,i_k\}}Y_{\{j_1,\dots,j_l,j\}}\Delta_{i,j+1}
\end{equation*}
and
\begin{equation*}
  x_{j+1}(X_{\{i_1,\dots,i_k\}}Y_{\{j_1,\dots,j_l\}}\Delta_{i,j}) + x_{i}X_{\{i_1,\dots,i_k\}}Y_{\{j_1,\dots,j_l\}}(\Delta_{j,j+1}) = X_{\{i_1,\dots,i_k,j\}}Y_{\{j_1,\dots,j_l\}}\Delta_{i,j+1}.
\end{equation*}
Note that the left hand side of both equalities is in $\mathfrak{a}$ by induction, which implies that the right hand side of both equalities is in $\mathfrak{a}$. As $\{i_1,\dots,i_k\} \cup \{j_1,\dots,j_l, j\} = \{i_1,\dots,i_k,j\} \cup \{j_1,\dots,j_l\} = [i+1, j]$, we are done. \end{proof}

In the following theorem we give an explicit computation for the generators of $\mathfrak{a}:I$.

\begin{thm}
  \label{generators of link}
  Let $\mathfrak{a} = (\{\Delta_{t,t+1}\:\vert \:1 \leq t \leq n-1\})$ and $M = \{X_KY_L \: \vert \: K\cup L = [ 2, n-1 ],\; K\cap L = \emptyset \}$. Then for any subset $\{m_1, \dots, m_{n-1}\}$ of $M$, where the bidegree of $m_i$ is not equal to the bidegree $m_j$ for any $i \neq j$, $J = \mathfrak{a}:I = \mathfrak{a}+ (\{m_1, \dots, m_{n-1}\})$.
\end{thm}

\begin{proof}
  Pick an $m_i$. Let $\Delta_{c,d}$ be a generator of $I$. Note that $c$ and $d$ are integers with $1 \leq c < d \leq n$. Note that by definition $m_i$ has a factor of the form $X_{K'}Y_{L'}$ such that $K' \cup L' = [c+1,d-1]$ and $K' \cap L' = \emptyset$. Thus, by \ref{in a}, $m_i \Delta_{c,d} \in \mathfrak{a}$. So $m_i \in J$. Thus $\mathfrak{a} + (\{m_1, \dots, m_{n-1}\}) \subseteq J$.

  The mapping cone construction tells us that a minimal generating set of $J$ has $n-1$ generators of degree 2 and $n-1$ generators of degree $n-2$, thus is we find linearly independent elements of $J$ then we have a generating set of $J$. Note that the set $\{\Delta_{1,2}, \dots, \Delta_{n-1,n}, m_1,\dots, m_{n-1}\}$ meets this criteria and is contained in $J$, so all that remains is to prove linear independence over $\mathbf{k}$.

  Note that every $\Delta_{i,i+1}$ is a homogeneous polynomial of bidegree $(1,1)$. Every $m_i$ has bidegree $(k,l)$ where $k+l = n-2$. The set $\{\Delta_{1,2}, \dots, \Delta_{n-1,n}\}$ is clearly linearly independent, so all that remains is to show that, for all $l$, $m_l \neq \sum_{i=1}^{n-1}b_i\Delta_{i,i+1} + \sum_{j=1, i \neq l}^{n-1} c_jm_j$.

  Suppose $m_l = \sum_{i=1}^{n-1}b_i\Delta_{i,i+1} + \sum_{j=1, i \neq l}^{n-1} c_jm_j$. As the $\Delta_{i,i+1}$ are homogeneous polynomials of degree 2 and the $m_j$ are monomials of degree $n-2$, and $m_l$ is a monomial of degree $n-2$, we may assume that the $b_i$ and $c_j$ are all homogeneous polynomials and moreover that that $c_j$ all have degree $0$. However, as by assumption the bidegree of $m_l$ is not equal to the bidegree of any of the $m_j$ where $j \neq l$, this implies $m_l = \sum_{i=1}^{n-1}b_i\Delta_{i,i+1}$.

  So, $m_l \in \mathfrak{a}\subseteq I$. Note that $I$ is prime \cite[Theorem 2.10]{BV}, so this implies $x_i$ or $y_i$ is in $I$ for some $i$. However $I$ is generated by homogeneous binomials of degree 2, so this is a contradiction. Thus $m_l \notin \mathfrak{a}$.

  Thus $m_l \neq \sum_{i=1}^{n-1}b_i\Delta_{i,i+1} + \sum_{j=1, i \neq l}^{n-1} c_jm_j$, and we are done. \end{proof}

\section{Special Case}\label{Sum of Links}

In this section we establish our special case and prove our main result, Theorem \ref{p2-main-result}, is true for this case. Then, in the following section, we will generalize the special case from Section \ref{Main Section} in order to prove Theorem \ref{p2-main-result}.

Let $\mathbf{k}$ be a field of characteristic 0. Let $R = \mathbf{k}[x_1,\dots,x_n,y_1,\dots,y_n,z_1,\dots,z_n]$, for an integer $n \geq 4$. Let $M = \begin{pmatrix}
  x_1 \dots x_n \\
  y_1 \dots y_n
\end{pmatrix}$ be a $2\times n$ generic matrix, and let $I \subsetneq R$ be the ideal generated by the $2\times 2$ minors of $M$. Recall that for our special case, we want to select  $\mathfrak{a} = (a_1,\dots,a_n)$ so that $\mathfrak{a}:I$ is an $n$-residual intersections, and, for all $1 \leq i \leq n$, $(a_1,\dots,\widehat{a_i},\dots,a_n):I$ is a geometric link. Thus, we begin by giving the notation used for the generators of $\mathfrak{a}$.
\begin{notation}\label{gens of a}
  Define the following binomials: $g_1 = z_1\Delta_{2,1},\; g_n = z_n\Delta_{n,n-1}$ and for\; $2 \leq i \leq n-1$, $g_i = z_i\Delta_{i+1,i-1}$. Let $\mathfrak{a} = (g_1,\dots,g_n)$ and let $\mathfrak{a}_i = (g_1, \dots, \widehat{g}_i, \dots, g_n)$. Let $J = \mathfrak{a}:I$ and $J_i = \mathfrak{a}_i:I$. 
\end{notation}
Using this notation, the specialization of Theorem \ref{p2-main-result} can be stated as:

\begin{restatable}{thm}{linksum}
\label{sum of links}
The ideal $J$ is a residual intersection equal to $\sum_{i=1}^n J_i$ and each $J_i$ is a link.
\end{restatable}

In order to prove the above theorem, we first establish that each $J_i$ is a link (Lemma \ref{a link}) and that $J$ is a residual intersection (Lemma \ref{is Res Int}). Once we have an explicit description of the generators of each $J_i$ (Lemma \ref{link is sum of monomials}), we use this to a compute a Gr\"obner basis of $\sum_{i=1}^n J_i$ (Lemma \ref{Grobner basis of J}). This computation allows us to prove Theorem \ref{sum of links}.

\begin{lem}
  \label{a link}
  Each $J_i$ is a link.
\end{lem}

\begin{proof}

In order to prove that $J_i$ is a link we only need that $\mathfrak{a}_i$ is a regular sequence. Let $R' = \mathbf{k}[x_1,\dots,x_n,y_1,\dots,y_n]$. Let $g_1' = \Delta_{2,1}$, $g_n' = \Delta_{n,n-1}$, and for $2 \leq i \leq n-1$, $g_i' = \Delta_{i+1,i-1}$. As the $z_j$ are indeterminates, it is enough to show that $\mathfrak{g}_i' = \{g_1', \dots, \widehat{g}_i',\dots,g_n'\}$ is a regular sequence over $R'$. By \ref{linkgens} it is enough to show that, up to sign, each $\mathfrak{g}_i'$ is obtained from $\{ \Delta_{t,t+1}\:\vert \:1 \leq t \leq n-1 \}$ via a ring automorphism.

First, for $\mathfrak{g}_n'$ we have the following ring automorphisms:

If $n$ is odd, let $m = \frac{n+1}{2}$ and define $\phi: R' \rightarrow R'$ by having $\phi$ permute the indices of the $x_i$ and $y_i$ as follows: If $i$ is odd, $i \mapsto m+ \frac{i-1}{2}$. If $i$ is even, $i \mapsto m- \frac{i}{2}$. Thus,
\begin{align*}
  \phi(\Delta_{2,1}) &= \Delta_{m-1,m};\\
  \phi(\Delta_{i+1,i-1}) &= -\Delta_{m + \frac{i}{2}-1,m + \frac{i}{2} } \quad \text{ if } i \text{ is even};\\
  \phi(\Delta_{i+1,i-1}) &= \Delta_{m - \frac{i+1}{2},m - \frac{i-1}{2}} \quad \text{ if } i \text{ is odd}.
\end{align*}

If $n$ is even, let $m = \frac{n}{2}$ and define $\phi: R' \rightarrow R'$ by having $\phi$ permute the indices of the $x_i$ and $y_i$ as follows: If $i$ is odd, $i \mapsto m - \frac{i-1}{2}$. If $i$ is even, $i \mapsto m + \frac{i}{2}$. Thus,
\begin{align*}
  \phi(\Delta_{2,1}) &= -\Delta_{m,m+1}&;\\
  \phi(\Delta_{i+1,i-1}) &= \Delta_{m - \frac{i}{2},m - \frac{i}{2} + 1 } \quad \text{ if } i \text{ is even};\\
  \phi(\Delta_{i+1,i-1}) &= -\Delta_{m + \frac{i-1}{2},m + \frac{i+1}{2}} \quad \text{ if } i \text{ is odd}.
\end{align*}

For $\mathfrak{g}_{n-1}'$ we have the following:

If $n$ is odd, let $m = \frac{n-1}{2}$ and define $\phi: R' \rightarrow R'$ by having $\phi$ permute the indicies of the $x_i$ and $y_i$ as follows: $n \mapsto n$. If $i$ is odd and $i \neq n$, $i \mapsto m - \frac{i-1}{2}$. If $i$ is even, $i \mapsto m + \frac{i}{2}$. Thus,
\begin{align*}
  \phi(\Delta_{2,1}) &= -\Delta_{m,m+1};\\
  \phi(\Delta_{n,n-1}) &= -\Delta_{n-1,n};\\
  \phi(\Delta_{i+1,i-1}) &= \Delta_{m - \frac{i}{2},m - \frac{i}{2} + 1 } \quad \text{ if } i \text{ is even};\\
  \phi(\Delta_{i+1,i-1}) &= -\Delta_{m + \frac{i-1}{2},m + \frac{i+1}{2}} \quad \text{ if } i \text{ is odd}.
\end{align*}

If $n$ is even, let $m = \frac{n}{2}$ and define $\phi: R' \rightarrow R'$ by having $\phi$ permute the indices of the $x_i$ and $y_i$ as follows: $n \mapsto n$. If $i$ is odd, $i \mapsto m + \frac{i-1}{2}$. If $i$ is even and $i \neq n$, $i \mapsto m - \frac{i}{2}$. Thus,
\begin{align*}
  \phi(\Delta_{2,1}) &= \Delta_{m-1,m};\\
  \phi(\Delta_{n,n-1}) &= -\Delta_{n-1,n};\\
  \phi(\Delta_{i+1,i-1}) &= -\Delta_{m + \frac{i}{2}-1,m + \frac{i}{2}} \quad \text{ if } i \text{ is even};\\
  \phi(\Delta_{i+1,i-1}) &= \Delta_{m - \frac{i+1}{2},m - \frac{i-1}{2}} \quad \text{ if } i \text{ is odd}.
\end{align*}

For $\mathfrak{g}_{i}'$, where $n-1<i<2$ we have the following:

If $n$ is even and $i$ is even, let $m = \frac{n}{2}$, let $k = \frac{i}{2}$, and define $\phi: R' \rightarrow R'$ by having $\phi$ permute the indices of the $x_j$ and $y_j$ as follows:
If $j$ is odd and $j\leq i-1$, $j \mapsto n-k + \frac{j+1}{2}$.
If $j$ is odd $j\geq i+1$, $j \mapsto \frac{j+1}{2} - k$.
If $j$ is even, $j \mapsto n-k+1 -\frac{j}{2}$. Thus,
\begin{align*}
  \phi(\Delta_{2,1}) &= \Delta_{n-k,n-k+1};\\
  \phi(\Delta_{n,n-1}) &= -\Delta_{m-k,m-k+1};\\
  \phi(\Delta_{j+1,j-1}) &= -\Delta_{n-k + \frac{j}{2}, n-k + \frac{j}{2}+1} \quad \text{ if } j \leq i-1 \text{ is even};\\
  \phi(\Delta_{j+1,j-1}) &= -\Delta_{\frac{j}{2}-k,\frac{j}{2}-k+1} \quad \text{ if } j \geq i+1 \text{ is even};\\
  \phi(\Delta_{j+1,j-1}) &= \Delta_{n+1 - k  - \frac{j+1}{2},n+1 - k - \frac{j-1}{2}} \quad \text{ if } j \text{ is odd}.
\end{align*}

If $n$ is even and $i$ is odd, let $m = \frac{n}{2}$, let $k = \frac{i-1}{2}$, and define $\phi: R' \rightarrow R'$ by having $\phi$ permute the indicies of the $x_j$ and $y_j$ as follows:
If $j$ is even and $j\leq i-1$, $j \mapsto n-k + \frac{j}{2}$.
If $j$ is even $j\geq i+1$, $j \mapsto \frac{j}{2} - k$.
If $j$ is odd, $j \mapsto n-k -\frac{j-1}{2}$. Thus,
\begin{align*}
  \phi(\Delta_{2,1}) &= \Delta_{n-k,n-k+1};\\
  \phi(\Delta_{n,n-1}) &= \Delta_{m-k,m-k+1};\\
  \phi(\Delta_{j+1,j-1}) &= \Delta_{n-k -\frac{j}{2},n-k - \frac{j}{2}+1} \quad \text{ if } j \text{ is even};\\
  \phi(\Delta_{j+1,j-1}) &= -\Delta_{n-k + \frac{j-1}{2}, n-k + \frac{j+1}{2}} \quad \text{ if } j \leq i-1 \text{ is odd};\\
  \phi(\Delta_{j+1,j-1}) &= -\Delta_{\frac{j-1}{2}-k,\frac{j+1}{2}-k} \quad \text{ if } j \geq i+1 \text{ is odd}.
\end{align*}

If $n$ is odd and $i$ is even, let $m = \frac{n+1}{2}$, let $k = \frac{i}{2}$, and define $\phi: R' \rightarrow R'$ by having $\phi$ permute the indices of the $x_j$ and $y_j$ as follows:
If $j$ is odd and $j\leq i-1$, $j \mapsto n-k + \frac{j+1}{2}$.
If $j$ is odd $j\geq i+1$, $j \mapsto \frac{j+1}{2} - k$.
If $j$ is even, $j \mapsto n-k+1 -\frac{j}{2}$. Thus,
\begin{align*}
  \phi(\Delta_{2,1}) &= \Delta_{n-k,n-k+1};\\
  \phi(\Delta_{n,n-1}) &= \Delta_{m-k,m-k+1};\\
  \phi(\Delta_{j+1,j-1}) &= -\Delta_{n-k + \frac{j}{2}, n-k + \frac{j}{2}+1} \quad \text{ if } j \leq i-1 \text{ is even};\\
  \phi(\Delta_{j+1,j-1}) &= -\Delta_{\frac{j}{2}-k,\frac{j}{2}-k+1} \quad \text{ if } j \geq i+1 \text{ is even};\\
  \phi(\Delta_{j+1,j-1}) &= \Delta_{n+1-k-\frac{j+1}{2},n+1-k-\frac{j-1}{2}} \quad \text{ if } j \text{ is odd}.
\end{align*}

If $n$ is odd and $i$ is odd, let $m = \frac{n-1}{2}$, let $k = \frac{i-1}{2}$, and define $\phi: R' \rightarrow R'$ by having $\phi$ permute the indices of the $x_j$ and $y_j$ as follows:
If $j$ is even and $j\leq i-1$, $j \mapsto n-k + \frac{j}{2}$.
If $j$ is even $j\geq i+1$, $j \mapsto \frac{j}{2} - k$.
If $j$ is odd, $j \mapsto n-k -\frac{j-1}{2}$. Thus,
\begin{align*}
  \phi(\Delta_{2,1}) &= -\Delta_{n-k,n-k+1};\\
  \phi(\Delta_{n,n-1}) &= -\Delta_{m-k,m-k+1};\\
  \phi(\Delta_{j+1,j-1}) &= \Delta_{n-k -\frac{j}{2},n-k - \frac{j}{2}+1} \quad \text{ if } j \text{ is even};\\
  \phi(\Delta_{j+1,j-1}) &= -\Delta_{n-k + \frac{j-1}{2}, n-k + \frac{j+1}{2}} \quad \text{ if } j \leq i-1 \text{ is odd};\\
  \phi(\Delta_{j+1,j-1}) &= -\Delta_{\frac{j-1}{2}-k,\frac{j+1}{2}-k} \quad \text{ if } j \geq i+1 \text{ is odd}.
\end{align*}

For $\mathfrak{g}_{2}'$ we have the following:

If $n$ is odd, let $m = \frac{n+1}{2}$ and define $\phi: R' \rightarrow R'$ by having $\phi$ permute the indices of the $x_i$ and $y_i$ as follows: $1 \mapsto 1$. If $i$ is odd and $i \neq 1$, $i \mapsto n + 1 - \frac{i-1}{2}$. If $i$ is even, $i \mapsto 1 + \frac{i}{2}$. Thus,
\begin{align*}
  \phi(\Delta_{2,1}) &= -\Delta_{1,2};\\
  \phi(\Delta_{n,n-1}) &= -\Delta_{m,m+1};\\
  \phi(\Delta_{i+1,i-1}) &= \Delta_{n+1-\frac{i}{2},n+2-\frac{i}{2}} \quad \text{ if } i \text{ is even};\\
  \phi(\Delta_{i+1,i-1}) &= -\Delta_{\frac{i-1}{2}+1,\frac{i+1}{2}+1} \quad \text{ if } i \text{ is odd}.
\end{align*}

If $n$ is even, let $m = \frac{n}{2}$ and define $\phi: R' \rightarrow R'$ by having $\phi$ permute the indices of the $x_i$ and $y_i$ as follows: $1 \mapsto 1$. If $i$ is odd and $i \neq 1$, $i \mapsto n + 1 - \frac{i-1}{2}$. If $i$ is even, $i \mapsto 1 + \frac{i}{2}$. Thus,
\begin{align*}
  \phi(\Delta_{2,1}) &= -\Delta_{1,2};\\
  \phi(\Delta_{n,n-1}) &= \Delta_{m+1,m+2};\\
  \phi(\Delta_{i+1,i-1}) &= \Delta_{n+1-\frac{i}{2},n+2-\frac{i}{2}} \quad \text{ if } i \text{ is even};\\
  \phi(\Delta_{i+1,i-1}) &= -\Delta_{\frac{i-1}{2}+1,\frac{i+1}{2}+1} \quad \text{ if } i \text{ is odd}.
\end{align*}

For $\mathfrak{g}_{1}'$ we have the following:

If $n$ is odd, let $m = \frac{n+1}{2}$ and define $\phi: R' \rightarrow R'$ by having $\phi$ permute the indices of the $x_i$ and $y_i$ as follows: If $i$ is odd, $i \mapsto 1 + \frac{i-1}{2}$. If $i$ is even, $i \mapsto n+1 - \frac{i}{2}$. Thus,
\begin{align*}
  \phi(\Delta_{n,n-1}) &= \Delta_{m,m+1};\\
  \phi(\Delta_{i+1,i-1}) &= -\Delta_{\frac{i}{2},\frac{i}{2}+1} \quad \text{ if } i \text{ is even};\\
  \phi(\Delta_{i+1,i-1}) &= \Delta_{n+1-\frac{i+1}{2},n+1-\frac{i-1}{2}} \quad \text{ if } i \text{ is odd}.
\end{align*}

If $n$ is even, let $m = \frac{n}{2}$ and define $\phi: R' \rightarrow R'$ by having $\phi$ permute the indices of the $x_i$ and $y_i$ as follows: If $i$ is odd, $i \mapsto 1 + \frac{i-1}{2}$. If $i$ is even, $i \mapsto n+1 - \frac{i}{2}$. Thus,
\begin{align*}
  \phi(\Delta_{n,n-1}) &= -\Delta_{m,m+1};\\
  \phi(\Delta_{i+1,i-1}) &= -\Delta_{\frac{i}{2},\frac{i}{2}+1} \quad \text{ if } i \text{ is even};\\
  \phi(\Delta_{i+1,i-1}) &= \Delta_{n+1-\frac{i+1}{2},n+1-\frac{i-1}{2}} \quad \text{ if } i \text{ is odd}.
\end{align*} \end{proof}

Note that in the proof of the above lemma, we applied ring automorphisms to the link from Section \ref{Sec Links} rather than using a similar proof to the proof of Lemma \ref{linkgens}. We do this so we can apply these same automorphisms in Lemma \ref{link is sum of monomials} to give an explicit description of the generators of each $J_i$.

However, before we can give an explicit description for the generators of each $J_i$, we must introduce further notation: Let $\{i_1,\dots,i_k\} \subseteq [ 1,n]$. Let $Z_{\{i_1,\dots,i_k\}} = z_{i_1}\dots z_{i_k}.$  Notice that this definition implies that $Z_{\{i_1,\dots,i_k\}}$ is a square free monomial. We define $Z_\emptyset$ as $1$.

\begin{notation}\label{monomials}
  Define $m_{i,j}$ as follows:
  \begin{equation*}
    m_{i,j} = \begin{cases}
    X_{[3,j-1]}Y_{[j,n]}Z_{[2,n]} & \text{if } i = 1 \text{ and } 3 \leq j \leq n+1;\\
    X_{[1,j-1]}Y_{[j,n-2]}Z_{[1,n-1]} & \text{if } i=n \text{ and } 1 \leq j \leq n-1;\\
    X_{[1,j-1]\setminus\{i-1,i+1\}}Y_{[j,n]\setminus\{i-1,i+1\}}Z_{[1,n]\setminus \{i\}} & \text{if } 2 \leq i \leq n-1$ \text{ and } $1 \leq j \leq n+1.
  \end{cases}
  \end{equation*}
\end{notation}

Note that for $i \notin \{1,n\}$, $m_{i,i} = m_{i,i-1}$ and $m_{i,i+1} = m_{i,i+2}$.

\begin{lem}
\label{link is sum of monomials}
Fix an integer $i$ such that $1 \leq i \leq n$. Then $J_i = \mathfrak{a}_i + (\{m_{i,j}\})$ for $j$ such that $3 \leq j \leq n+1\,$ if $\,i = 1$, $1 \leq j \leq n-1\,$ if $\,i = n$, and $\,1 \leq j \leq n+1\,$ otherwise.
\end{lem}
\begin{proof}
  Let $g_1' = \Delta_{2,1}$, $g_n' = \Delta_{n,n-1}$ and for $2 \leq i \leq n-1$, $g_i' = \Delta_{i+1,i-1}$. And let $\mathfrak{a}_i' = (g_1', \dots, \widehat{g}_i',\dots,g_n')$, $J_i' = a_i':I$.

  Let $M = \begin{cases}
      \{X_KY_L \: \vert \: K\cup L = [3,n], K\cap L = \emptyset \}& \text{if } \:i = 1 \\
      \{X_KY_L \: \vert \: K\cup L = [1,n-2],K\cap L = \emptyset \}& \text{if }\: i = n \\
      \{X_KY_L \: \vert \: K\cup L = [1, n] \setminus\{i+1,i-1\}, K\cap L = \emptyset \}&  \text{if }\: 2 \leq i \leq n-1.
   \end{cases}$

  Using the maps $\phi$ from the proof of \ref{a link} we that $\phi$ maps the indices $\{1,2\}$ to $\{1,n\}$ if $i=1$, $\phi$ maps the indices $\{n-1,n\}$ to $\{1,n\}$ if $i=n$, and $\phi$ maps the indices $\{i+1,i-1\}$ to $\{1,n\}$ otherwise. Thus, applying \ref{generators of link}, we get that for any subset $\{m_1 \dots m_{n-1}\}$ of $M$ where the bidegre of $m_i$ is not equal to the bidegree $m_j$ for any $i\neq j$, $J_i' = \mathfrak{a}_i'+ (m_1, \dots, m_{n-1})$.

  Thus, $J_i' = \begin{cases}
      \mathfrak{a}_i'+(\{X_{[3,j-1]}Y_{[j,n]}\:\vert \:  3 \leq j \leq n+1\})& \text{if }\: i = 1 \\
      \mathfrak{a}_i' + (\{X_{[1,j-1]}Y_{[j,n-2]}\:\vert \: 1 \leq j \leq n-1\})& \text{if }\: i = n \\
      \mathfrak{a}_i'+ (\{X_{[1,j-1\}\setminus\{i-1,i+1\}}Y_{[j,n]\setminus\{i-1,i+1\}}\:\vert \:1 \leq j \leq n+1\}) & \text{if }\: 2 \leq i \leq n-1.
   \end{cases}$

   Note that in the last case $j = i$ gives the same element as $j = i-1$, similarly $j=i+1$ gives the same element as $j = i+2$.

   Let $\underline{g}_i = g_1, \dots, \widehat{g}_i,\dots,g_n$ and $\underline{g}_i' = g_1', \dots, \widehat{g}_i',\dots,g_n'$. Let $(F_\bullet, \lambda_\bullet)$ be a minimal, homogeneous, free $R$-resolution of $R/I$. Note that $I$ is a perfect ideal \cite[Corollary 2.8]{BV}, so the length of $F_\bullet$ is $n-1$. Let $K_\bullet(\underline{g}_i)$ be the Koszul complex of $\underline{g}_i$ and $K_\bullet(\underline{g}_i')$ be the Koszul complex of $\underline{g}_i'$. Let $e_j$ be the basis element of $K_1(\underline{g}_i)$ that maps to $g_j$ and $e_j'$ be the basis element of $K_1(\underline{g}_i')$ that maps to $g_j'$.

   Define $\gamma_\bullet \: : \: K_\bullet(\underline{g}_i) \rightarrow K_\bullet(\underline{g}_i')$ so that $\gamma_1: e_j \mapsto z_je_j'$ and $\gamma_k = \wedge^k \gamma_1$. Define $u_\bullet' \: : \: K_\bullet(\underline{g}_i') \rightarrow F_\bullet\;$  to be any  morphism of complexes such that $u_0'$ is the identity map on $R$. Define $ u_\bullet \: : \: K_\bullet(\underline{g}_i) \rightarrow F_\bullet\ $ as $u_\bullet = u'_\bullet \circ \gamma_\bullet$.

  Let $h = \grade(I) = n-1$. Since $I$ is a perfect ideal \cite[Corollary 2.8]{BV}, we may use the mapping cone construction to see that $J_i = (\underline{g}_i) + I_1(u_h) = \mathfrak{a}_i+I_1(u_h)$ and $J_i' = (\underline{g}_i')+I_1(u_h') = \mathfrak{a}_i'+I_1(u_h')$. We see that
  \begin{equation*}
    u_h = \wedge^h (u_1'\circ \gamma_1) = \det(\gamma_1)u_h',
  \end{equation*}
 so
 \begin{equation*}
   I_1(u_h) = det(\gamma_1)I_1(u_h') = Z_{[1,n]\setminus\{i\}}I_1(u_h').
 \end{equation*}

   Thus it follows that
   \begin{equation*}
     J_i = \begin{cases}
         \mathfrak{a}_i +(\{X_{[3,j-1]}Y_{[j,n]}Z_{[2,n]}\:\vert \:  3 \leq j \leq n+1\})& \text{if }\: i = 1 \\
         \mathfrak{a}_i + (\{X_{[1,j-1]}Y_{[j,n-2]}Z_{[1,n-1]}\:\vert \: 1 \leq j \leq n-1\})& \text{if }\: i = n \\
         \mathfrak{a}_i+ (\{X_{[1,j-1\}\setminus\{i-1,i+1\}}Y_{[j,n]\setminus\{i-1,i+1\}}Z_{[1,n]\setminus\{i\}}\:\vert \:1 \leq j \leq n\}) & \text{if }\: 2 \leq i \leq n-1.
      \end{cases}
   \end{equation*} \end{proof}

\begin{cor}\label{set-M}
  For $1 \leq i \leq n$, define the set $M_i$ as follows:
  \begin{equation*}
    M_i = \begin{cases}
        \{X_KY_LZ_{[2,n]} \: \vert \: K\cup L = [3,n], K\cap L = \emptyset \}& \text{if } \:i = 1 \\
        \{X_KY_LZ_{[1,n-2]} \: \vert \: K\cup L = [1,n-2],K\cap L = \emptyset \}& \text{if }\: i = n \\
        \{X_KY_LZ_{[1,n]\setminus\{i\}} \: \vert \: K\cup L = [1, n] \setminus\{i+1,i-1\}, K\cap L = \emptyset \}&  \text{if }\: 2 \leq i \leq n-1.
     \end{cases}
  \end{equation*}
  Then, $J_i = \mathfrak{a}_i+(M_i)$.
\end{cor}

The proof of the above corollary follows directly from the proof of Lemma \ref{link is sum of monomials}. We state this corollary, as the Gr\"obner basis of $\sum_i^nJ_i$ which we compute contains $\bigcup_{i=1}^n M_i$. While it is possible to further reduce this Gr\"obner basis, the computation of a reduced Gr\"obner basis requires extra work and is unnecessary for the proof of our main result.

\begin{lem}\label{geometric link}
  Each $J_i$ is a geometric link.
\end{lem}

\begin{proof}
  By Lemma \ref{a link}, each $J_i$ is link, moreover by Lemma \ref{link is sum of monomials} each $J_i$ contains a monomial. Note that $I$ is a prime ideal \cite[Theorem 2.10]{BV} generated by degree two polynomials. Thus, $I$ cannot contain a monomial, as if $I$ contained a monomial, it would have to contain a degree one polynomial, which is a contradiction. So $J_i \not\subseteq I$, and since $I$ is prime, that means that $\height(I+J_i) \geq \height(I)+1$, equivalently $J_i$ is a geometric link. \end{proof}

Now that we have computed the generators of $J_i$ and proven that each $J_i$ is a geometric link, we must show that $J$ is an $n$-residual intersection. In order to do so, we will use graph theoretic results about binomial edge ideals from \cite{HHHKR}. Thus, we must briefly review some terminology and introduce some notation.

We call a graph \emph{simple} if it has no loops and no multiple edges. Let $G$ be a simple graph on the vertex set $[1,n]$. The \emph{binomial edge ideal} of $G$, $J_G$, is the ideal generated by $\Delta_{i,j}$ where $\{i,j\}$ is an edge in $G$. Notice that $\mathfrak{a}$ and $I$, as well as $\mathfrak{a}_i$ for all $1 \leq i \leq n$, are binomial edge ideals. Thus we may apply the results from \cite{HHHKR} to our ideals.

\begin{lem}\label{is Res Int}
  The height of the ideal $J_n + (g_n)$ is at least $n$.
\end{lem}

\begin{proof}
As $J_n = \mathfrak{a}_n:I$ is geometric link by Lemma \ref{geometric link}, and $I$ is a prime ideal \cite[Theorem 2.10]{BV}, $\Ass(J_n) = \Ass(\mathfrak{a}_n)\backslash \{I\}$. Since $\mathfrak{a}_n$ is a generated by a regular sequence, the set of its associated primes is equivalent to the set of its minimal primes. Thus, all we need to show is that the $g_n$ is not contained in any minimal prime $p \neq I$ of $\mathfrak{a}_n$.

Let $g_1' = \Delta_{2,1}$, $g_n' = \Delta_{n,n-1}$ and for $2 \leq i \leq n-1$, $g_i' = \Delta_{i+1,i-1}$. Let $p \neq I$ be a minimal prime of $\mathfrak{a}_n$ and let $T \subseteq [1,n-1]$ be exactly the set of indices such that $z_i$, for $i \in T$, is in $p$. Note that we are allowing $T$ to be the empty set. Now we localize $R$ at $p$. Note that $z_i$, for $i \notin T$, becomes a unit in $R_p$. So, $p_p = (\{z_i \; \vert  \; i \in T\}) + p'_p$ where $p'$ is a minimal prime of $\mathfrak{b} = (\{g_i' \; \vert  \; i \in [1,n-1]\backslash T\})$. Note that $\mathfrak{b} = J_G$, the binomial edge ideal for a simple graph $G$.

Let $\mathfrak{a}_n'  = (g_1',\dots,g_{n-1}')$. It is easy to see that $\mathfrak{a}_n' = J_F$ is the binomial edge ideal for a simple graph $F$ where $F$ is a line with endpoints $n$ and $n-1$. Note that $G$ is a subgraph of this line.

By \cite[Theorem 3.2]{HHHKR} $p' = P_S(G) = (\bigcup_{i\in S}\{x_i, y_i\}, J_{G_1'},\dots,J_{G_t'})$ for some $S \subseteq [1,n]$ where $G_{[1,n-1]\backslash S}$ is the restriction of $G$ to $[1,n-1]\backslash S$ whose edges are exactly those edges $\{i,j\}$ of $G$ for which $\{i,j\} \subseteq [1,n-1]\backslash S$; $G_1, \dots, G_t$ are the connected components of $G_{[1,n-1]\backslash S}$; and $G_i'$ is the complete graph of the vertex set $V(G_i)$.

Suppose $n$ or $n-1$ is in $S$. Let $S' = S\backslash \{n,n-1\}$. Note, by \cite[Theorem 3.2]{HHHKR} $\mathfrak{b} \subset P_{S'}(G)$. Let $G_1,\dots,G_t$ be the connected components of $G_{[1,n-1]\backslash S}$ and $H_1, \dots H_{t'}$ be the connected components of $G_{[1,n]\backslash S'}$. Note that since $n$ and $n-1$ are on the endpoints of the line segments contained in the graph $G$, removing them does not split any connected components.

Thus, for all $i \in [1,t']$, $V(H_i)\backslash S \subseteq V(G_j)$ for some $j \in [1,t]$. So, by \cite[Proposition 3.8]{HHHKR} $P_{S'}(G) \subsetneq P_{S}(G)$, which is a contradiction as $p'$ is a minimal prime of $\mathfrak{b}$. So $n$ and $n-1$ are not contained in $S$.

For $n$ and $n-1$ to be in the vertex set of the same connected component of $G_{[1,n-1]\backslash S}$, both $S$ and $T$ have to be the empty set, but in that case $P_{S}(G) = I$, so $p_p = I_p$, thus $p = I$, which is a contradiction. So $g_n'$ is not a generator of $p'$, and $z_n \notin p$, thus $g_n \notin p$ and we are done. \end{proof}

The above lemma not only implies that $J$ is an $n$-residual intersection, but it also gives us that $\height(\sum_{i=1}^n J_i)\geq n$.

\subsection{Gr\"obner Basis}\label{sec-p2-gb}

Now, we will compute a Gr\"obner basis of $\sum_{i=1}^n J_i$. The monomial ordering we use to do so is reverse lexicographical order with $x_1>x_2>\dots>x_n>y_1>y_2>\dots>y_n>z_1>z_2>\dots>z_n$.  

Here we will briefly review some facts about Gr\"obner bases, but we refer the reader to \cite{HH} for a more detailed exploration. 

The \emph{initial term} of a polynomial $f$ is defined to the term with the largest monomial and denoted as $\initial(f)$. For an $R$-ideal $I$ we define the \emph{initial ideal} of $I$ as the ideal generated by the initial terms of all the polynomials in $I$,  $\initial(I) = (\{ \initial(f) \;\vert\; f \in I\})$. Critical to the results of this paper is the fact that, if $\initial(I)$ is a squarefree monomial ideal, then $I$ is reduced \cite[Proposition 3.3.7]{HH}.

The Gr\"obner basis is used to compute the initial ideal. A set of nonzero $R$-polynomials $G$ is said to be a \emph{Gr\"obner basis} of an $R$-ideal $I$ if  $\initial(I) = (\{\initial(g) \;\vert\; g \in G\})$. In order to calculate a Gr\"obner basis, we will be using Buchberger’s Criterion. To give Buchberger’s Criterion, we first need to define the support of a polynomial and $S$-polynomials, as well as give the division algorithm.

If we let $M$ be the set of all monomials in $R$, then any polynomial $f$ can be written as a unique $\mathbf{k}$-linear combination of those monomials, $f = \sum_{u \in M}c_u u$, and we define the \emph{support} of $f$ as $\support(f) = \{ u \in N \;\vert\; c_u \neq 0\}$.
\begin{thm}\cite[Theorem 2.2.1]{HH}
  Let $f_1,f_2,\dots,f_n$ be nonzero polynomials of $R$. Then, for any nonzero polynomial $h \in R$, there exist polynomials $h_1,h_2,\dots,h_n$ and $h'$ of $R$ such that
  \begin{equation*}
    h = h'+ \sum_{i=1}^n h_if_i,
  \end{equation*}
  where:
  \begin{enumerate}[label=$($\alph*$)$]
    \item if $h' \neq 0$ and $u \in \support(h')$, then $u \notin (\initial(f_1),\initial(f_2),\dots,\initial(f_n));$
    \item if $h_i \neq 0$, then $\initial(h) \geq \initial(h_if_i)$.
  \end{enumerate}
\end{thm}
In this paper we will define the \emph{gcd} of the two terms to be the greatest common divisor with coefficient 1.

 For any two nonzero polynomials $f$ and $g$ of $R$, the polynomial \begin{equation*}
    S(f,g) = \frac{\initial(g)}{\gcd(\initial(f),\initial(g))}f - \frac{\initial(f)}{\gcd(\initial(f),\initial(g))}g
  \end{equation*} is called the \emph{$S$-polynomial} of $f$ and $g$. If, with respect to $R$-polynomials $f_1, \dots, f_n$, the remainder in the division algorithm is zero, we say that $f$ \emph{reduces to 0} with respect to $f_1, \dots, f_n$.

\begin{thm}\cite[Theorem 2.3.2]{HH} (Buchberger's criterion).  Let $I= (f_1,\dots,f_n)$ be a nonzero $R$-ideal. Then $\{f_1,\dots,f_n\}$ is a Gr\"obner basis of I if and only if, for all $i \neq j$, $S(f_i,f_j)$ reduces to 0 with respect to $f_1, \dots, f_n$.
\end{thm}

Now, we will begin computing the Gr\"obner basis of $\sum_{i=1}^n J_i$. 

\begin{notation}\label{more gs}
  Let
  \begin{equation*}
    g_{1,j} = X_{[1,j-1]}Z_{[1,j]}\Delta_{j+1,j} \text{ for } 1 \leq j \leq n-1
  \end{equation*} and
  \begin{equation*}
    g_{j,n} = Y_{[j+1,n]} Z_{[j,n]}\Delta_{j,j-1} \text{ for } 2 \leq j \leq n.
  \end{equation*}
\end{notation}

Note that the distinction between $g_{1,j}$ and $g_{j,n}$ is clear as neither permits the index $({1,n})$. Also note that $g_{1,1} = g_1$ and $g_{n,n} = g_n$.

\begin{notation}\label{set G}
  Let $G = \{g_i \; \vert  \; 1 \leq i \leq n\} \cup \{g_{1,i} \;\vert \: 2 \leq i \leq n-1 \}\cup \{g_{i,n} \;\vert \; 2 \leq i \leq n-1 \}$ for $g_i$ as in \ref{gens of a}, as well as $g_{1,i}$ and $g_{i,n}$ as in \ref{more gs}.
\end{notation}

We will show that $G\cup(\bigcup_{i=1}^n M_i)$ for $M_i$ as in Corollary \ref{set-M} is a Gr\"obner basis of $\sum_{i=1}^n J_i$, however, in order to do so we first argue that $G$ is a Gr\"obner basis of $\mathfrak{a}$, as this allows for more legible proofs. To simply our computations, we begin by proving Lemma \ref{lemma 1}, Lemma \ref{lemma 2}, and Corollary \ref{lemma 3}.

\begin{lem}\label{lemma 1}
For integers $i$ and $j$ such that $1 \leq i < j\leq n$,
\begin{equation*}
  X_{[1,j-1]\setminus\{i\}}Z_{[1,j-1]}\Delta_{j,i} = \sum_{g_\alpha \in G} f_\alpha g_\alpha
\end{equation*} such that, for all $\alpha$, $g_\alpha$ is in $G$, and $\initial(X_{[1,j-1]\setminus\{i\}}Z_{[1,j-1]}\Delta_{j,i}) = X_{[1,j]\setminus\{i\}}Y_{\{i\}}Z_{[1,j-1]} \geq \initial(f_\alpha g_\alpha)$.
\end{lem}

\begin{proof}

Note that for $i \leq k \leq j-1$,
\begin{equation*}
  X_{[k+2,j]}Z_{[k+1,j-1]}g_{1,k} = X_{[1,j]\setminus\{k\}}Y_{\{k\}}Z_{[1,j-1]}  -  X_{[1,j]\setminus\{k+1\}}Y_{\{k+1\}}Z_{[1,j-1]},
\end{equation*}
which gives us,
\begin{equation*}
  X_{[1,j-1]\setminus\{i\}}Z_{[1,j-1]}\Delta_{j,i} = \sum_{k=i}^{j-1}X_{[k+2,j]}Z_{[k+1,j-1]}g_{1,k}.
\end{equation*}
Moreover, for $i \leq k \leq j-1$, we have that,
\begin{eqnarray*}
  \initial(X_{[k+2,j]}Z_{[k+1,j-1]}g_{1,k}) &=& X_{[1,j]\setminus\{k\}}Y_{\{k\}}Z_{[1,j-1]} \\
  &\leq& X_{[1,j]\setminus\{i\}}Y_{\{i\}}Z_{[1,j-1]}\\
  &=& \initial(X_{[1,j-1]\setminus\{i\}}Z_{[1,j-1]}\Delta_{j,i}).
\end{eqnarray*}
Thus, we are done. \end{proof}

\begin{lem}\label{lemma 2}
  For integers $i$ and $j$ such that $1 \leq j < i \leq n$,
  \begin{equation*}
    Y_{[j+1,n]\setminus\{i\}}Z_{[j+1,n]}\Delta_{i,j} = \sum_{g_\alpha \in G} f_\alpha g_\alpha,
  \end{equation*} such that, for all $\alpha$, $g_\alpha$ is in $G$, and $\initial(Y_{[j+1,n]\setminus\{i\}}Z_{[j+1,n]}\Delta_{i,j}) = X_{\{i\}} Y_{[j,n]\setminus\{i\}} Z_{[j+1,n]} \geq \initial(f_\alpha g_\alpha)$.
\end{lem}

\begin{proof}
Note that for that for $j+1 \leq k \leq i$,
\begin{equation*}
  Y_{[j,k-2]}Z_{[j+1,k-1]}g_{k,n} = X_{\{k\}} Y_{[j,n]\setminus\{k\}} Z_{[j+1,n]} - X_{\{k-1\}} Y_{[j,n]\setminus\{k-1\}} Z_{[j+1,n]},
\end{equation*}
which gives us,
\begin{equation*}
  Y_{[j+1,n]\setminus\{i\}} Z_{[j+1,n]}\Delta_{i,j} = \sum_{k=j+1}^{i} Y_{[j,k-2]}Z_{[j+1,k-1]}g_{k,n}.
\end{equation*}
Moreover, for $j+1 \leq k \leq i$, we have that,
\begin{eqnarray*}
  \initial(Y_{[j,k-2]}Z_{[j+1,k-1]}g_{k,n}) &=& X_{\{k\}} Y_{[j,n]\setminus\{k\}} Z_{[j+1,n]} \\
  &\leq& X_{\{i\}} Y_{[j,n]\setminus\{i\}} Z_{[j+1,n]}\\
  &=& \initial(Y_{[j+1,n]\setminus\{i\}}Z_{[j+1,n]}\Delta_{i,j}).
\end{eqnarray*}
Thus, we are done.\end{proof}

The following lemma further simplifies our computation of a Gr\"obner basis of $\sum_{i=1}^n J_i$.

\begin{lem}
\label{general lemma 3}
  Let $f$ and $g$ be two binomials which are not monomials, such that $f \neq g$. Let $c = \gcd(\initial(f),\initial(g))$ and suppose $c$ divides $\initial(f)-f$ and $\initial(g)-g$. Then $S(f,g)$ reduces to 0 with respect to $f,g$.
\end{lem}

\begin{proof}
  Without loss of generality we can assume $f = f_1-f_2$ where $f_1$ and $f_2$ are terms where $\initial(f) = f_1$. Similarly we can assume $g = g_1-g_2$ where $g_1$ and $g_2$ are terms where $\initial(g) = g_1$. Thus, we have the following:

  \begin{eqnarray*}
    \frac{\initial(f)}{c}g-\frac{\initial(g)}{c}f &=&
   \frac{f_1}{c}(g_1-g_2) - \frac{g_1}{c}(f_1-f_2) \\
    &=&\frac{f_2}{c}g_1 - \frac{g_2}{c}f_1 \\
    &=& \frac{f_2}{c}g - \frac{g_2}{c}f.
  \end{eqnarray*}

  Note that if $\frac{\initial(f)}{c}g-\frac{\initial(g)}{c}f  = 0$, then we are done. Otherwise we have that $\frac{f_2}{c}g_1 \neq \frac{g_2}{c}f_1$. Note that $\initial(\frac{f_2}{c}g) = \frac{f_2}{c}\initial(g) = \frac{f_2}{c}g_1 \leq \initial(\frac{f_2}{c}g_1 - \frac{g_2}{c}f_1)$. Similarly, $\initial(\frac{g_2}{c}f) = \frac{g_2}{c}\initial{f} = \frac{g_2}{c}f_1 \leq \initial(\frac{f_2}{c}g_1 - \frac{g_2}{c}f_1)$. So, $S(f,g)$ reduces to zero with respect to $f$ and $g$. \end{proof}

\begin{cor}
\label{lemma 3}
  Let $\{g_\alpha, g_\beta\} \subseteq G$, for $G$ from Notation \ref{set G}. If $g = \gcd(\initial(g_\alpha),\initial(g_\beta))$ divides $\initial(g_\alpha) - g_\alpha$ and $\initial(g_\beta) - g_\beta$, then $S(g_\alpha,g_\beta)$ reduces to 0 with respect to $g_\alpha, g_\beta$.
\end{cor}

\begin{proof}
Follows directly from \ref{general lemma 3}.\end{proof}

For organizational purposes, we first compute a Gr\"obner basis for $\mathfrak{a}$.

\begin{lem}
  \label{Grobner basis of a}
  The set $G$ from Notation \ref{set G} is a Gr\"obner basis of $\mathfrak{a}$.
\end{lem}

\begin{proof}
  First note that $\mathfrak{a} = (g_1,\dots,g_n)$, thus $\mathfrak{a}$ is contained in the ideal generated by $G$. Also note that, for $1 \leq i \leq n-1$, $X_{[1,i-1]\cup\{i+1\}}Z_{[1,i]}g_{i+1}-z_{i+1}x_{i+2}g_{1,i} =  g_{1,i+1}$ and, for $2 \leq j \leq n$, $Y_{\{j-1\}\cup[j+1,n]}Z_{[j,n]}g_{j-1}-z_{j-1}x_{j}g_{n,j} = g_{n,j-1}$. As $g_{1,1} = g_1$ and $g_{n,n} = g_n$, it follows that the ideal generated by $G$ is contained in $\mathfrak{a}$. Thus $G$ is a generating set of $\mathfrak{a}$.

  Let $\{g_\alpha, g_\beta\} \subseteq G$. First note that $g_\alpha$ has the form $c\Delta_{j,i}$ where $j = i+1$ or $j= i+2$ and $c$ is a monomial. Similarly, $g_\beta$ has the form $d\Delta_{l,k}$ where $l = k+1$ or $l=k+2$ and $d$ is a monomial. Furthermore, note that $g_\alpha$ and $g_\beta$ are the sum of square free monomials.

  Let  $g = \gcd(\initial(g_\alpha),\initial(g_\beta))$. Let \begin{equation*}
    f = S(g_\beta, g_\alpha) =  \frac{\initial(g_\alpha)}{g}g_\beta - \frac{\initial(g_\beta)}{g}g_\alpha = \frac{cd}{g}(X_{\{i,l\}}Y_{\{j,k\}}-X_{\{j,k\}}Y_{\{i,l\}}).
    \end{equation*}

  \emph{Case 1}: $j = l$ and $i = k$. Then, $f = 0$.

  \emph{Case 2}: Suppose that $j=l$ and $i \neq k$. Without loss of generality we may assume that $i = j-2$ and $k = j-1$. Then we have
  \begin{equation*}
    g_\alpha = z_{i+1}\Delta_{i+2,i}
  \end{equation*} and
  \begin{equation*}
    g_\beta = d\Delta_{i+2,i+1}.
  \end{equation*} So
  \begin{eqnarray*}
    g &=& \gcd(z_{i+1}x_{i+2}y_i,dx_{i+2}y_{i+1})\\
        &=& \gcd(z_{i+1}x_{i+2}y_i,d)\gcd(z_{i+1}x_{i+2}y_i,x_{i+2})\gcd(z_{i+1}x_{i+2}y_i,y_{i+1})\\
        &=& \gcd(z_{i+1}x_{i+2}y_i,d)x_{i+2} \\
        &=& \gcd(z_{i+1},d)\gcd(x_{i+2},d)\gcd(y_{i},d)x_{i+2} \\
        &=& \gcd(z_{i+1},d)\gcd(y_i,d)x_{i+2}.
  \end{eqnarray*}

  The first and final equalities follow from the fact that $\initial(g_\beta)$ is square free and thus $x_{i+2}$ and $y_{i+1}$ do not divide $d$.

  Note that $d = Y_{[i+3,n]}Z_{[i+2,n]}$ or $d = X_{[1,i]}Z_{[1,i+1]}$, thus $\gcd(y_i,d)=1$. So we have $g = \gcd(z_{i+1},d)x_{i+2}$. Thus, $f = \lcm(z_{i+1},d)(y_{i+1}x_{i}y_{i+2} - x_{i+1}y_{i}y_{i+2})= \lcm(z_{i+1},d)y_{i+2}\Delta_{i,i+1}$. Note that $\lcm(z_{i+1},d) = Y_{[i+3,n]}Z_{[i+1,n]}$ or $\lcm(z_{i+1},d) = X_{[1,i]}Z_{[1,i+1]}$, so $f = -g_{i+1,n}$ or $f = -y_{i+2}x_iz_{i+1}g_{1,i}$.

  \emph{Case 3}: Suppose $j \neq l$ but $i = k$. Without loss of generality assume that $j>l$, then
  \begin{equation*}
    g_\alpha = z_{i+1}\Delta_{i+2,i}
  \end{equation*} and \begin{equation*}
    g_\beta = d\Delta_{i+1,i}.
\end{equation*}
  \begin{eqnarray*}
    g &=& \gcd(z_{i+1}x_{i+2}y_i,dx_{i+1}y_{i})\\
        &=& \gcd(z_{i+1}x_{i+2}y_i,d)\gcd(z_{i+1}x_{i+2}y_i,x_{i+1})\gcd(z_{i+1}x_{i+2}y_i,y_{i})\\
        &=& \gcd(z_{i+1}x_{i+2}y_i,d)y_{i}\\
        &=& \gcd(z_{i+1},d)\gcd(x_{i+2},d)\gcd(y_i,d)y_i\\
        &=& \gcd(z_{i+1},d)\gcd(x_{i+2},d)y_i.
  \end{eqnarray*}

  The first and last equalities follow from the fact that $\initial(g_\beta)$ is square free and thus $y_i$ and $x_{i+1}$ do not divide $d$.

  Note that $d = Y_{[i+2,n]} Z_{[i+1,n]}$ or $d = X_{[1,i-1]} Z_{[1,i]} $, thus $\gcd(x_{i+2},d)=1$. So we have $g = \gcd(z_{i+1},d)y_{i}$. Thus, $f = \lcm(z_{i+1},d)(x_{i+1}x_iy_{i+2}-x_iy_{i+1}x_{i+2}) = \lcm(z_{i+1},d)x_i\Delta_{i+1,i+2}$. Note that $\lcm(z_{i+1},d) = Y_{[i+2,n]} Z_{[i+1,n]}$ or $\lcm(z_{i+1},d) = X_{[1,i-1]}Z_{[1,i+1]}$, so $f = -g_{1,i+1}$ or $f = -x_iy_{i+2}z_{i+1}g_{i+2,n}$.

  \emph{Case 4}: Suppose $j \neq l$ and $i \neq k$.

  \begin{eqnarray*}
    g &=& \gcd(cx_{j}y_i,dx_{l}y_{k})\\
    &=& \gcd(cx_{j}y_i,d)\gcd(cx_{j}y_i,x_{l}y_{k})\\
    &=& \gcd(c,d)\gcd(x_{j}y_i,d)\gcd(c,x_{l}y_{k})\gcd(x_{j}y_i,x_{l}y_{k})\\
    &=& \gcd(c,d)\gcd(x_{j}y_i,d)\gcd(c,x_{l}y_{k}).
  \end{eqnarray*}

  Note, that if $g$ divides $\initial(g_\alpha) - g_\alpha$ and $\initial(g_\beta) - g_\beta$, then we are done by Corollary \ref{lemma 3}. So, without loss of generality, assume that $g$ does not divide $\initial(g_\alpha) - g_\alpha$. Note $\initial(g_\alpha) = cx_jy_i$ and $\initial(g_\alpha) - g_\alpha = cx_iy_j$. Since $g$ divides $cx_jy_i$ and $cx_jy_i$ is a square free monomial, then $x_j$ divides $g$ or $y_i$ divides $g$ (as if neither do, then $g$ must divide $c$ and thus divide $cx_iy_j$).

  \emph{Case 4.a}:

  First assume that $x_j$ divides $g$. Note that $x_j$ dividing $g$ implies that $x_j$ divides $\initial(g_\beta)$. Since $x_j$ divides $\initial(g_\beta)$ and does not divide $x_ly_k$, $x_j$ must divide $d$. Thus, $d = X_{[1,k-1]}Z_{[1,k]}$
  and $j \leq k-1$. Also note that this implies $l = k+1$.

  We have three options for $g_\alpha$

\begin{enumerate}[label=$($\alph*$)$]
  \item $g_\alpha = z_{i+1}\Delta_{i+2,i}$ where $i \leq k-3$
  \item $g_\alpha = X_{[1,i-1]}Z_{[1,i]}\Delta_{i+1,i}$ where $i \leq k-2$
  \item $g_\alpha = Y_{[i+2,n]}Z_{[i+1,n]}\Delta_{i+1,i}$ where $i \leq k-2$.
\end{enumerate}

\emph{In cases a and b}, $g = \gcd(c,d)x_j$, thus $f = \frac{\lcm(c,d)}{x_j}(X_{\{i,k+1\}}Y_{\{j,k\}}-X_{\{j,k\}}Y_{\{i,k+1\}})$. In this case $\lcm(c,d)= X_{[1,k-1]}Z_{[1,k]}$. So we have,
  \begin{eqnarray*}
    f &=& X_{[1,k-1]\setminus \{j\}}X_{\{i,k+1\}} Y_{\{j,k\}} Z_{[1,k]} - X_{[1,k]}Y_{\{i,k+1\}}Z_{[1,k]} \\
    &=& X_{[1,k-2]\setminus\{j\}}X_{\{i,k+1\}}Y_{\{k\}}Z_{[1,k]}\Delta_{k-1,j}+X_{[1,k-2]}X_{\{i\}}Y_{\{k\}}Z_{[1,k]}\Delta_{k+1,k-1}\\
    &&-X_{[1,k-1]}Y_{\{k+1\}}Z_{[1,k]}\Delta_{k,i}\\
    &=&X_{\{i,k+1\}}Y_{\{k\}}Z_{\{k-1,k\}}(X_{[1,k-2]\setminus\{j\}}Z_{[1,k-2]}\Delta_{k-1,j}) + X_{[1,k-2]}X_{\{i\}}Y_{\{k\}}Z_{[1,k-1]}g_k\\
    && -X_{\{i\}}Y_{\{k+1\}}Z_{\{k\}}(X_{[1,k-1]\setminus\{i\}}Z_{[1,k-1]}\Delta_{k,i})\\
    & =& X_{[1,k-2]}X_{\{i\}}Y_{\{k\}}Z_{[1,k-1]} g_k + \sum_{g_\gamma \in G}f_\gamma g_\gamma.
  \end{eqnarray*}

The final equality follows from Lemma \ref{lemma 1}. Note that,
\begin{equation*}
  \initial(X_{[1,k-1]}Y_{\{k+1\}}Z_{[1,k]}\Delta_{k,i}) = X_{[1,k]}Y_{\{i,k+1\}}Z_{[1,k]} \leq \initial(f),
\end{equation*} and, as $j \leq k-1$,
\begin{eqnarray*}
  \initial(X_{[1,k-2]}X_{\{i\}}Y_{\{k\}}Z_{[1,k]}\Delta_{k+1,k-1}) &=&X_{[1,k-1]\setminus\{k-1\}}X_{\{i,k+1\}}Y_{\{k-1,k\}}Z_{[1,k]}\\
  &\leq& X_{[1,k-1]\setminus \{j\}}X_{\{i,k+1\}} Y_{\{j,k\}}Z_{[1,k]}\\
  &\leq& \initial(f).
\end{eqnarray*}
Also note that, if $j \neq k-1$,
\begin{equation*}
  \initial(X_{[1,k-2]\setminus\{j\}}X_{\{i,k+1\}}Y_{\{k\}}Z_{[1,k]}\Delta_{k-1,j}) = X_{[1,k-1]\setminus \{j\}}X_{\{i,k+1\}} Y_{\{j,k\}} \leq \initial(f).
\end{equation*}
Furthermore, note that by Lemma \ref{lemma 1}, for all $\gamma$, \begin{equation*}
  \initial(g_\gamma f_\gamma) \leq \max\{\initial(X_{[1,k-2]\setminus\{j\}}X_{\{i,k+1\}}Y_{\{k\}}Z_{[1,k]}\Delta_{k-1,j}), \initial(X_{[1,k-1]}Y_{\{k+1\}}Z_{[1,k]}\Delta_{k,i})\}.
\end{equation*}

\emph{In case c}, $g= gcd(c,d)x_{i+1}y_k$.  $f = \frac{\lcm(c,d)}{x_{i+1}y_k}(X_{\{i,k+1\}}Y_{\{i+1,k\}}-X_{\{i+1,k\}}Y_{\{i,k+1\}})$. Note in this case $\lcm(c,d)= X_{[1,k-1]}Y_{[i+2,n]}Z_{[1,n]}$. So we have,
\begin{eqnarray*}
  f &=& X_{[1,k-1]\setminus\{i+1\}}X_{\{i,k+1\}} Y_{[i+1,n]}Z_{[1,n]} - X_{[1,k]} Y_{[i+2, n]\setminus\{k\}}Y_{\{i,k+1\}}Z_{[1,n]} \\
  &=& X_{[1,k-2]\setminus\{i+1\}}X_{\{i,k+1\}}Y_{[i+2,n]}Z_{[1,n]}\Delta_{k-1,i+1} + X_{[1,k-2]}X_{\{i\}}Y_{[i+2,n]}Z_{[1,n]}\Delta_{k+1,k-1}\\
  & & -X_{[1,k-1]}Y_{[i+3,n]\setminus\{k\}}Y_{\{i,k+1\}}Z_{[1,n]}\Delta_{k,i+2} - X_{[1,k-1]}Y_{[i+3,n]}Y_{\{k+1\}}Z_{[1,n]}\Delta_{i+2,i} \\
  &=& X_{\{i,k+1\}}Y_{[i+2,n]}Z_{[k-1,n]}(X_{[1,k-2]\setminus\{i+1\}}Z_{[1,k-2]}\Delta_{k-1,i+1}) + X_{[1,k-2]}X_{\{i\}} Y_{[i+2,n]}Z_{[1,n]\setminus \{k\}} g_k \\
  & & - X_{[1,k-1]}Y_{\{i,k+1\}}Z_{[1,i+2]}(Y_{[i+3,n]\setminus\{k\}}Z_{[i+3,n]}\Delta_{k,i+2}) - X_{[1,k-1]} Y_{[i+3,n]}Y_{\{k+1\}} Z_{[1,n]\setminus\{i+1\}}g_{i+1}\\
  &=& X_{[1,k-2]}X_{\{i\}} Y_{[i+2,n]}Z_{[1,n]\setminus \{k\}} g_k - X_{[1,k-1]} Y_{[i+3,n]}Y_{\{k+1\}} Z_{[1,n]\setminus\{i+1\}}g_{i+1} \\
  &&+\:\: \sum_{g_\gamma \in G}f_\gamma g_\gamma.
\end{eqnarray*}

The final equality follows from Lemma \ref{lemma 1} and Lemma \ref{lemma 2}. Note that, as $i < k-2$,
\begin{eqnarray*}
  \initial(X_{[1,k-2]}X_{\{i\}} Y_{[i+2,n]}Z_{[1,n]\setminus \{k\}} g_k) &=&  X_{[1,k-1]\setminus\{k-1\}}X_{\{i,k+1\}}Y_{\{k-1\}}Y_{[i+2,n]}Z_{[1,n]}\\
  &\leq& X_{[1,k-1]\setminus\{i+1\}}X_{\{i,k+1\}} Y_{[i+1,n]}Z_{[1,n]}\\
  &\leq& \initial(f),
\end{eqnarray*} and
\begin{eqnarray*}
  \initial(X_{[1,k-1]} Y_{[i+3,n]}Y_{\{k+1\}} Z_{[1,n]\setminus\{i+1\}}g_{i+1}) &=& X_{[1,k-1]}X_{\{i+2\}}Y_{[i+3,n]}Y_{\{i,k+1\}}Z_{[1,n]} \\
  &\leq&   X_{[1,k]} Y_{[i+2, n]\setminus\{k\}}Y_{\{i,k+1\}}Z_{[1,n]} \\
  &\leq& \initial(f).
\end{eqnarray*}
Moreover, if $i+2\neq k$,
\begin{equation*}
  \initial(X_{[1,k-2]\setminus\{i+1\}}X_{\{i,k+1\}}Y_{[i+2,n]}Z_{[1,n]}\Delta_{k-1,i+1}) = X_{[1,k-1]\setminus\{i+1\}}X_{\{i,k+1\}} Y_{[i+1,n]}Z_{[1,n]} \leq \initial(f)
\end{equation*}
and
\begin{equation*}
  \initial(X_{[1,k-1]}Y_{[i+3,n]\setminus\{k\}}Y_{\{i,k+1\}}Z_{[1,n]}\Delta_{k,i+2}) = X_{[1,k]} Y_{[i+2, n]\setminus\{k\}}Y_{\{i,k+1\}}Z_{[1,n]} \leq \initial(f).
\end{equation*}
Furthermore, note that by Lemma \ref{lemma 1} and Lemma \ref{lemma 2}, for all $\gamma$,
\begin{multline*}
    \initial(f_\gamma g_\gamma)  \leq \max\{\initial(X_{[1,k-2]\setminus\{i+1\}}X_{\{i,k+1\}}Y_{[i+2,n]}Z_{[1,n]}\Delta_{k-1,i+1}), \\ \initial(X_{[1,k-1]}Y_{[i+3,n]\setminus\{k\}}Y_{\{i,k+1\}}Z_{[1,n]}\Delta_{k,i+2})\}.
\end{multline*}

\emph{Case 4.b}:

  Now assume that $y_i$ divides $g$. Note that $y_i$ dividing $g$ implies that $y_i$ divides $\initial(g_\beta)$. Since $y_i$ divides $\initial(g_\beta)$ and as $y_i$ does not divide $x_ly_k$, $y_i$ must divide $d$. Thus, $d = Y_{[l+1,n]}Z_{[l,n]}$
  and $i \geq l+1$. Also note that this implies $k = l-1$ and that $j>l+1$.

  We have three options for $g_\alpha$

  \begin{enumerate}[label=$($\alph*$)$]
    \item $g_\alpha = z_{i+1}\Delta_{i+2,i}$
    \item $g_\alpha = X_{[1,i-1]}Z_{[1,i]}\Delta_{i+1,i}$
    \item $g_\alpha = Y_{[i+2,n]}Z_{[i+1,n]}\Delta_{i+1,i}$.
  \end{enumerate}

  \emph{In cases a and c}, $g = \gcd(c,d)y_i$, thus $f = \frac{\lcm(c,d)}{y_i}(X_{\{l,i\}}Y_{\{l-1,j\}}-X_{\{l-1,j\}}Y_{\{l,i\}})$. Note in this case $\lcm(c,d)= Y_{[l+1,n]}Z_{[l,n]}$. So we have,
  \begin{eqnarray*}
    f &=& X_{\{l,i\}}Y_{[l+1,n]\setminus\{i\}}Y_{\{l-1,j\}}Z_{[l,n]} - X_{\{l-1,j\}}Y_{[l,n]}Z_{[l,n]} \\
    &=&  X_{\{l\}}Y_{[l+2,n]\setminus \{i\}}Y_{\{l-1,j\}}Z_{[l,n]}\Delta_{i,l+1} + X_{\{l\}}Y_{[l+2,n]}Y_{\{j\}}Z_{[l,n]}\Delta_{l+1,l-1} - X_{\{l-1\}}Y_{[l+1,n]}Z_{[l,n]}\Delta_{j,l}\\
    &=&  X_{\{l\}}Y_{\{l-1,j\}}Z_{\{l,l+1\}}(Y_{[l+2,n]\setminus \{i\}}Z_{[l+2,n]}\Delta_{i,l+1}) + X_{\{l\}}Y_{[l+2,n]}Y_{\{j\}}Z_{[l+1,n]}g_l \\
    & & - X_{\{l-1\}}Y_{\{j\}}Z_{\{l\}}(Y_{[l+1,n]\setminus\{j\}}Z_{[l+1,n]}\Delta_{j,l})\\
    &=&  X_{\{l\}}Y_{[l+2,n]}Y_{\{j\}}Z_{[l+1,n]}g_l  + \sum_{g_\gamma \in G}f_\gamma g_\gamma.
  \end{eqnarray*}

  The final equality follows from Lemma \ref{lemma 2}. Note that,
  \begin{equation*}
    \initial( X_{\{l-1\}}Y_{[l+1,n]}Z_{[l,n]}\Delta_{j,l}) = X_{\{l-1,j\}}Y_{[l,n]}Z_{[l,n]} \leq \initial(f),
  \end{equation*}
  and, as $l+1\leq i$,
  \begin{eqnarray*}
    \initial(X_{\{l\}}Y_{[l+2,n]}Y_{\{j\}}Z_{[l+1,n]}g_l) & = & X_{\{l,l+1\}}Y_{[l+1,n]\setminus\{l+1\}}Y_{\{l-1,j\}}Z_{[l,n]} \\
    &\leq&X_{\{l,i\}}Y_{[l+1,n]\setminus\{i\}}Y_{\{l-1,j\}}Z_{[l,n]}\\
    &\leq& \initial(f).
  \end{eqnarray*}
  Also note that, if $i \neq l+1$,
  \begin{equation*}
    \initial( X_{\{l\}}Y_{[l+2,n]\setminus \{i\}}Y_{\{l-1,j\}}Z_{[l,n]}\Delta_{i,l+1}) = X_{\{l,i\}}Y_{[l+1,n]\setminus\{i\}}Y_{\{l-1,j\}}Z_{[l,n]} \leq \initial(f).
  \end{equation*}

  Furthermore, note that by Lemma \ref{lemma 2}, for all $\gamma$,
  \begin{equation*}
    \initial(f_\gamma g_\gamma) \leq \max\{\initial( X_{\{l-1\}}Y_{[l+1,n]}Z_{[l,n]}\Delta_{j,l}), \initial( X_{\{l\}}Y_{[l+2,n]\setminus \{i\}}Y_{\{l-1,j\}}Z_{[l,n]}\Delta_{i,l+1})\}.
  \end{equation*}

  \emph{In case b}, $g= gcd(c,d)x_{l}y_i$.  $f = \frac{\lcm(c,d)}{x_{l}y_i}(X_{\{l,i\}}Y_{\{l-1,i+1\}}-X_{\{l-1,i+1\}}Y_{\{l,i\}})$. Note in this case $\lcm(c,d)= X_{[1,i-1]}Y_{[i+1,n]}Z_{[1,n]}$. So we have,
  \begin{eqnarray*}
    f &=& X_{[1,i]}Y_{[l+1,n]\setminus\{i\}}Y_{\{l-1,i+1\}}Z_{[1,n]} -  X_{[1,i-1]\setminus \{l\}}X_{\{l-1,i+1\}}Y_{[l,n]}Z_{[1,n]}\\
    &=& X_{[1,i-1]}Y_{[l+2,n]\setminus\{i\}}Y_{\{l-1,i+1\}}Z_{[1,n]}\Delta_{i,l+1} +  X_{[1,i-1]}Y_{[l+2,n]}Y_{\{i+1\}}Z_{[1,n]}\Delta_{l+1,l-1}\\
    & & - X_{[1,i-2]\setminus \{l\}}X_{\{l-1,i+1\}}Y_{[l+1,n]}Z_{[1,n]}\Delta_{i-1,l} - X_{[1,i-2]}X_{\{l-1\}}Y_{[l+1,n]}Z_{[1,n]}\Delta_{i+1,i-1}\\
    &=& X_{[1,i-1]}Y_{\{l-1,i+1\}}Z_{[1,l+1]}(Y_{[l+2,n]\setminus\{i\}}Z_{[l+2,n]}\Delta_{i,l+1}) +  X_{[1,i-1]}Y_{[l+2,n]}Y_{\{i+1\}}Z_{[1,n]\setminus\{l\}}g_l\\
    & & - X_{\{l-1,i+1\}}Y_{[l+1,n]}Z_{[i-1,n]}(X_{[1,i-2]\setminus \{l\}}Z_{[1,i-2]}\Delta_{i-1,l}) - X_{[1,i-2]}X_{\{l-1\}}Y_{[l+1,n]}Z_{[1,n]\setminus\{i\}}g_i\\
    &=&  X_{[1,i-1]}Y_{[l+2,n]}Y_{\{i+1\}}Z_{[1,n]\setminus\{l\}}g_l -  X_{[1,i-2]}X_{\{l-1\}}Y_{[l+1,n]}Z_{[1,n]\setminus\{i\}}g_i \\
    &&+\:\: \sum_{g_\gamma \in G}f_\gamma g_\gamma
  \end{eqnarray*}
The final equality follows from Lemma \ref{lemma 1} and Lemma \ref{lemma 2}. Note that, as $l+1 \leq i$,
\begin{eqnarray*}
  \initial(X_{[1,i-1]}Y_{[l+2,n]}Y_{\{i+1\}}Z_{[1,n]\setminus\{l\}}g_l) &=& X_{[1,i-1]}X_{\{l+1\}}Y_{[l+1,n]\setminus\{l+1\}}Y_{\{l-1,i+1\}}Z_{[1,n]}\\
  &\leq& X_{[1,i]}Y_{[l+1,n]\setminus\{i\}}Y_{\{l-1,i+1\}}Z_{[1,n]} \\
  &\leq& \initial(f),
\end{eqnarray*} and
\begin{eqnarray*}
  \initial(X_{[1,i-2]}X_{\{l-1\}}Y_{[l+1,n]}Z_{[1,n]\setminus\{i\}}g_i) &=& X_{[1,i-2]}X_{\{l-1,i+1\}}Y_{\{i-1\}}Y_{[l+1,n]}Z_{[1,n]}\\
  &\leq& X_{[1,i-1]}X_{\{l-1\}}Y_{[l+1,n]}Y_{\{i+1\}}Z_{[1,n]}\\
  &\leq& \initial(f).
\end{eqnarray*}
Moreover, if $i\neq l+1$,
\begin{equation*}
  \initial(X_{[1,i-1]}Y_{[l+2,n]\setminus\{i\}}Y_{\{l-1,i+1\}}Z_{[1,n]}\Delta_{i,l+1}) = X_{[1,i]}Y_{[l+1,n]\setminus\{i\}}Y_{\{l-1,i+1\}}Z_{[1,n]} \leq \initial(f),
\end{equation*} and
\begin{equation*}
  \initial(X_{[1,i-2]\setminus \{l\}}X_{\{l-1,i+1\}}Y_{[l+1,n]}Z_{[1,n]}\Delta_{i-1,l})= X_{[1,i-1]\setminus \{l\}}X_{\{l-1,i+1\}}Y_{[l,n]}Z_{[1,n]} \leq \initial(f).
\end{equation*}
Furthermore, note that by Lemma \ref{lemma 1} and Lemma \ref{lemma 2}, for all $\gamma$,
\begin{multline*}
      \initial(f_\gamma g_\gamma)  \leq \max\{\initial(X_{[1,i-1]}Y_{[l+2,n]\setminus\{i\}}Y_{\{l-1,i+1\}}Z_{[1,n]}\Delta_{i,l+1}),\\ \initial(X_{[1,i-2]\setminus \{l\}}X_{\{l-1,i+1\}}Y_{[l+1,n]}Z_{[1,n]}\Delta_{i-1,l})\}.
\end{multline*}
\end{proof}

\begin{lem}
  \label{Grobner basis of J}
  Let $M = \bigcup_{i=1}^n$, for $M_i$ as in Corollary \ref{set-M}. A Gr\"obner basis of\: $\sum_{i=1}^n J_i$ is $G \cup M$, for $G$ as in Notation \ref{set G}.
\end{lem}
\begin{proof}
  It follows from the proof of Lemma \ref{Grobner basis of a} and from Corollary \ref{set-M} that $G\cup M$ is a generating set of $\sum_{i=1}^n J_i$. Note that from the proof of Lemma \ref{Grobner basis of a} any $\{g_\alpha, g_\beta\} \subseteq G$, $S(g_\alpha, g_\beta)$ reduces to 0 with respect to $G$. Let $m_\alpha$ and $m_\beta$ be elements of $M$, and let $g = \gcd(\initial(m_\alpha),\initial(m_\beta))$. Note that $S(m_\alpha,m_\beta) = \frac{\initial(m_\alpha)}{g} m_\beta - \frac{\initial(m_\alpha)}{g} m_\beta=0$. So, we only have to check the $S$-polynomials of $g_\alpha \in G$ and $m_\beta \in M$.

  Let $g = \gcd(\initial(g_\alpha),\initial(m_\beta))$. Notice that, as $\initial(g_\alpha)$ and $\initial(m_\beta)$ are square free, so is $g$. Let \begin{equation*}
    f = S(g_\alpha,m_\beta) = \frac{\initial(m_\beta)}{g} g_\alpha - \frac{\initial(g_\alpha)}{g} m_\beta = \frac{\initial(g_\alpha)-g_\alpha}{g}m_\beta.
  \end{equation*}

  If $g$ divides $\initial(g_\alpha)-g_\alpha$, then we are done. So, suppose $g$ does not divide $\initial(g_\alpha)-g_\alpha$.  We can write $g_\alpha$ as $c(x_ly_k - x_ky_l)$ where $l= k+1$ or $l=k+2$. As $g$ is square free, $x_l,y_k,x_k$ and $y_l$ do not divide $c$. Since $g$ does not divide $\initial(g_\alpha)-g_\alpha$, $g$ does not divide $cx_ky_l$, thus $x_l$ or $y_k$ divide $g$ and $m_\beta$.

  Also, note that for all $x_s$, $y_t$ that divide $m_\beta$, by definition, $s \neq t$. Thus if $x_l$ divides $m_\beta$ then $y_l$ does not divide $m_\beta$. Similarly if $y_k$ divides $m_\beta$ then $x_k$ does not divide $m_\beta$. Also note that $m_\beta$ is square free.

  Thus $m_\beta = \begin{cases} a x_l \text{ where } x_l \text{ and } y_l \text{ do not divides } a &\quad \text{if } x_l \text{ divides } m_\beta \\
  a y_k \text{ where } x_k \text{ and } y_k \text{ do not divides } a & \quad \text{if } y_k \text{ divides } m_\beta.
  \end{cases}$

  If $x_l$ divides $m_\beta$, then:
  \begin{eqnarray*}
    g & = &\gcd(cx_ly_k,ax_l) \\
    &=& \gcd(cx_ly_k,a)\gcd(cx_ly_k,x_l) \\
    &=& \gcd(cx_ly_k,a)x_l \\
    &=& \gcd(c,a)\gcd(x_ly_k,a)x_l \\
    &=& \gcd(c,a)\gcd(y_k,a)x_l.
  \end{eqnarray*}

  If $y_k$ divides $m_\beta$, then:
  \begin{eqnarray*}
    g & = &\gcd(cx_ly_k,ay_k) \\
    &=& \gcd(cx_ly_k,a)\gcd(cx_ly_k,y_k) \\
    &=& \gcd(cx_ly_k,a)y_k \\
    &=& \gcd(c,a)\gcd(x_ly_k,a)y_k \\
    &=& \gcd(c,a)\gcd(x_l,a)y_k.
  \end{eqnarray*}

Thus, letting $b = \frac{c}{\gcd(c,a)}$,
\begin{equation*}
  f = \begin{cases}
  bx_k(\frac{y_l m_\beta}{x_l}) &\quad \text{if } x_l \text{ divides } m_\beta \text{ and } y_k \text{ does not}\\
  b y_l(\frac{x_k m_\beta }{y_k}) &\quad \text{if } y_k \text{ divides } m_\beta \text{ and } x_l \text{ does not}\\
  b(\frac{x_k y_l m_\beta}{x_l y_k}) & \quad \text{otherwise}.
  \end{cases}
\end{equation*}

Notice, that for some $i$, $m_\beta \in M_i \subseteq M$, which implies that $m_\beta = X_LY_KZ_{[1,n]\setminus\{i\}}$ with $L \cap K = \emptyset$. If $x_l$ divides $m_\beta$, then $l \in L$, so
  \begin{equation*}
    \frac{y_lm_\beta}{x_l} = X_{L\setminus\{l\}}Y_{K\cup\{l\}}Z_{[1,n]\setminus\{i\}}.
  \end{equation*} As $(L\setminus\{l\})\cap(K\cup\{l\}) = \emptyset$ and $(L\setminus\{l\})\cup(K\cup\{l\})=L\cup K$, we have that $\frac{y_lm_\beta}{x_l} \in M_i \subseteq M$. If $y_k$ divides $m_\beta$, $k \in K$, so we have that
    \begin{equation*}
      \frac{x_k m_\beta}{y_k} = X_{L\cup\{k\}}Y_{K\setminus\{k\}}Z_{[1,n]\setminus\{i\}}.
    \end{equation*} As
    $(L\cup\{k\})\cap(K\setminus\{k\}) = \emptyset$ and
    $(L\cup\{k\})\cup(K\setminus\{k\})=L\cup K$, we have that $\frac{x_k m_\beta}{y_k} \in M_i \subseteq M$. If $x_l$ and $y_k$ divide $m_\beta$, then $l \in L$ and $k \in K$, so we have that
      \begin{equation*}
        \frac{x_ky_lm_\beta}{x_ly_k} = X_{(L\cup\{k\})\setminus\{l\}}Y_{(K\cup\{l\})\setminus\{k\}}Z_{[1,n]\setminus\{i\}}.
      \end{equation*} As
      $((L\cup\{k\})\setminus\{l\})\cap((K\cup\{l\})\setminus\{k\}) = \emptyset$ and
      $((L\cup\{k\})\setminus\{l\})\cup((K\cup\{l\})\setminus\{k\})=L\cup K$, we have that $\frac{x_ky_lm_\beta}{x_ly_k} \in M_i \subseteq M$. Thus $f$ reduces zero with respect to $G \cup M$, and we are done. \end{proof}

\subsection{Main Result in the Special Case}

Now, we have all the tools we need to prove the main result of this section.

\linksum*

\begin{proof}
  By Lemma \ref{a link} each $J_i$ is a geometric link and by Lemma \ref{is Res Int} $J$ is a residual intersection. Let $J' = \sum_{i=1}^n J_i$. Note, since $J' \subseteq J$, it is enough to show that $J_p = J_p'$ for all the associated primes $p$ of $J'$. We begin by arguing that the set of associated primes of $J'$ is just the set of minimal primes of $J'$.

  From Lemma \ref{Grobner basis of J} we have $G \cup M$ is a Gr\"obner basis of $J'$. By definition, $\{\initial(g_\alpha) \; \vert  \; g_\alpha \in G \cup M\}$ generates $\initial(J')$.  Note that all elements in $G$ are sums of square free monomials and all elements in $M$ are square free monomials. Thus, we can see that $\initial(J')$ is square free and thus $J'$ is reduced \cite[Proposition 3.3.7]{HH}. So, the set of associated of $J'$ is the set of minimal primes of $J'$.

  Let $m = (x_1,\dots,x_n,y_1,\dots,y_n)$. Note that $m$ is a prime ideal of $R$. We now argue that $\support(J/J')_R \subseteq V(m)$. Suppose that $q$ is a prime that does not contain $m$, then $I_q$ is a complete intersection, so $J_q = J'_q$ \cite[Corollary 2.18]{KMU}. Thus $q \notin \support(J/J')$.

  All that is left to show is that no minimal prime of $J'$ contains $m$. We now argue that $J \subseteq m$. Suppose $J$ is not contained in $m$, then $J_m = R_m$, which implies that $\mathfrak{a}_m = I_m$, which further implies that $I_m$ is generated by at most $n$ elements. However, we know that $\mu(I_m) = \binom{n}{2} > n.$ This is a contraction, so we have $J' \subseteq J \subseteq m$.

  Suppose there exists a minimal prime $p$ of $J'$ such that $m \subseteq p$. Since $p$ is a minimal prime of $J'$ and $m$ contains $J'$, $p = m$. Since $J$ is between $J'$ and $m$, $m$ is a minimal prime of $J$. Notice $I$ satisfies $G_{2n}$, and notice that $I$ is Cohen-Macaulay, which implies its links are also Cohen-Macaulay \cite{PS}. So, we may apply \cite[Proposition 1.7]{U} every $n$-residual intersection of $I$ is unmixed of height $n$. So $J$ is unmixed of height $n$. But $\height(m) = 2n$, so $m$ cannot be a minimal prime of $J$. Thus, no minimal primes of $J'$ contain $m$ and we are done. \end{proof}

\section{Main Result}\label{Main Section}

Now that we have shown that our main result is true in a special case, all that is left is to generalize. In order to do so, we will first need to show that the residual intersections we are working with commute with surjective maps (Theorem \ref{general specialization}). Note that, in general, colon ideals do not commute with surjective maps. Once we have shown our residual intersections commute with surjective maps, we will use this result to show that a generic $n$-residual intersection of an ideal generated by the maximal minors of a generic $2\times n$ matrix is the sum of links (Lemma \ref{MainLem}). This will allow us to apply Theorem \ref{general specialization} to a generic matrix to get our main result.

\begin{thm}
  \label{general specialization}
  Let $R$ be a Noetherian local ring, $\pi:S \rightarrow \overline{S}$ be a surjection of local $R$-algebras whose kernel is generated by an $S$-regular sequence, $\underline{l}$, and assume that the map $R \rightarrow S$ is local. Assume $S$ is a Gorenstein ring and that $\overline{S}$ is flat over $R$. Let $I'$ be an $R$-ideal that has height $h \leq n$, $I$ be the extension of $I'$ to $S$, and $J = \mathfrak{a}:I$ be an $S$-ideal such that $\mathfrak{a} \subset I$ is an $n$-generated $S$-ideal.

  Suppose $I'$ is $G_n$, $R/I'$ is Cohen-Macaulay, and $\depth(R/I'^j) \geq \dim R-h-j+1$ for $1 \leq j \leq n-h$. Also suppose that $\Ext^{n+1}_R(R/(I')^{n-h+1},R)=0$. If $\pi(\mathfrak{a}):\pi(I)$ is an $n$-residual intersection then:
  \begin{enumerate}[label=$($\arabic*$)$]
    \item $\pi(J) = \pi(\mathfrak{a}):\pi(I)$
    \item $\underline{l}$ is a regular sequence on $S/J$.
  \end{enumerate}
\end{thm}

\begin{proof}

  Let $\underline{l} = l_1,\dots,l_m$ and let $\underline{l}_k = l_1,\dots,l_k$. Let define $\pi_k$ to be the natural surjections $S \rightarrow S/(\underline{l}_k)$. By \cite[Theorem 22.6]{M2} $S/(\underline{l}_k)$ is flat over $R$. By \cite[Lemma 4.1]{HU}, we have that $\height(\pi_k(\mathfrak{a}):\pi_k(I))\geq n$. So, we may reduce to the case where $\underline{l}$ is a single element. Let said element be $l$. Thus, we can reduce to $\overline{S} = S/(l)$. Notice that both $S$ and $\overline{S}$ are flat over $R$.

  Because the maps $R \rightarrow S$ and $R \rightarrow \overline{S}$ are flat and local, and the rings $S$ and $\overline{S}$ are local Gorenstein rings, the special fibers are Gorenstein \cite[Theorem 23.4]{M2}. Therefore, again as $S$ and $\overline{S}$ are flat $R$-algebras, we have that $\Ext^{n+1}_S(S/I^{n-h+1},S)=0$ and $\Ext^{n+1}_{\overline{S}}(\overline{S}/\overline{I}^{n-h+1},\overline{S})=0$. Also, by flatness, we have the going down property, thus $I$ and $\overline{I}$ are $G_n$.

  Let $M$ be $R/(I')^j$ for $1 \leq j \leq n-h$, let $m$ be the maximal ideal of $R$, and let $F = S \otimes R/m$. By flatness we have that $\depth_S(M\otimes_R S) = \depth_R(M) + \depth_S(F)$ \cite[Theorem 50]{M1} and that $\dim_S(M\otimes_R S) = \dim_R(M) + \dim_S(F)$ \cite[Theorem 51ii]{M1}. As noted previously, $F$ is Gorenstein, and therefore we have that $S/I$ is Cohen-Macaulay, and $\depth(S/I^j) \geq \dim S-h-j+1$ for $1 \leq j \leq n-h$. A similar argument shows that $\overline{S}/\overline{I}$ is Cohen-Macaulay, and $\depth(\overline{S}/\overline{I}^j) \geq \dim \overline{S}-h-j+1$ for $1 \leq j \leq n-h$.

  Without loss of generality, by adjoining an indeterminate to $S$ and localizing, we may assume $S$ has an infinite residue field.

  We will prove our theorem by induction on $n$. Suppose that $n = h$. By assumption we have that $\mu(\mathfrak{a}) = h$. As $\pi(\mathfrak{a}):\pi(I)$ is an $h$-residual intersection, the height of $\height(\pi(\mathfrak{a})) = h$. So, we have the inequality $h \geq \height(\mathfrak{a}) \geq \height(\pi(\mathfrak{a})) = h$. This implies that $\mathfrak{a}$ is a complete intersection ideal, and thus unmixed. Since $\height(\mathfrak{a}) = \height(\pi(\mathfrak{a}))$, $l$ is not contained in any minimal prime of $\mathfrak{a}$, and thus is not contained in any associated prime of $\mathfrak{a}$. Thus $l$ is regular on $S/\mathfrak{a}$, and (1) and (2) follow from \cite[Proposition 2.8]{HU-old}. 

  Now suppose $n > h$. Let $\overline{\cdot}$ represent images mod $l$. Note that, as $\overline{I}$ is $G_{n}$, by \cite[Corollary 1.6a]{U} we can select general elements $\underline{a} = a_1,\dots,a_n$  that are generators of $\mathfrak{a}$ such that for $\mathfrak{a}_k = (a_1,\dots,a_k)$, $\overline{\mathfrak{a}_k}:\overline{I}$ is a geometric $k$-residual intersection for all $k <n$. For notational purposes write $J_k = \mathfrak{a}_k : I$.  Note, $\overline{J_k} = \overline{\mathfrak{a}_k}:\overline{I}$ for all $h \leq k <n$ by the induction hypothesis. As for all $k<n$, $\height(J_{k}) \geq \height(\overline{J_k})$ and  $\height(I + J_{k}) \geq \height(\overline{I+J_k})$, $J_k$ is a geometric $k$-residual intersection.

  Let $S_k = S/J_k$ and $A = S_{n-1}$. As the depth properties of $I'$ pass to $I$, and since $S$ is Gorenstein, $I$ is $AN_{n-1}$ \cite[Theorem 2.9]{U}. In particular, $A$ is Cohen-Macaulay. Let $\omega_A$ be the canonical module of $A$. Note that if an ideal is contained in $S$, $\overline{\cdot}$ represents images in $\overline{S}$, and if an ideal is contained in $A$, $\overline{\cdot}$ represents images in $\overline{A} = A/lA$.

  By induction $l$ is regular on $A$ and thus we have the exact sequence
  \begin{equation*}
    0 \rightarrow \omega_A \xrightarrow{\cdot l} \omega_A \rightarrow \omega_A/l\omega_A \rightarrow 0.
  \end{equation*}

  We will first show that, after applying the functor $\Hom_A(I^{n-h+1}A,-)$, we have the short exact sequence
  \begin{equation}\label{homEq}
    0 \rightarrow \Hom_A(I^{n-h+1}A,\omega_A) \xrightarrow{\cdot l} \Hom_A(I^{n-h+1}A,\omega_A) \rightarrow \Hom_A(I^{n-h+1}A,\omega_A/l\omega_A) \rightarrow 0.
  \end{equation}

  For this it is enough to show that $\Ext_A^1(I^{n-h+1}A,\omega_A) = 0.$

  As $A$ is Cohen-Macaulay, we have that $\Ext_A^1(I^{n-h+1}A,\omega_A) \cong \Ext_S^{n}(I^{n-h+1}A,\omega_S)$ and since $S$ is Gorenstein, $\omega_S \cong S$. Note that, again as $S$ is Gorenstein, $\depth(S/I^k) \geq \dim S-h-k+1$ implies that $\Ext^j(S/I^k,S) = 0$ for all $j \geq h+k$. Thus $\Ext_{S}^{h+k}(S/I^k,S)=0$ for $1 \leq k \leq n-h+1$. From the proof of \cite[Theorem 4.1]{CEU}, we have that  $\Ext_S^{n}(I^{n-h+1}A,S) = 0$. For the convenience of the reader we will reprove the relevant portion of the proof of \cite[Theorem 4.1]{CEU} and use our stronger assumptions to simplify it.

  In order to do this, we will show that given $i$ such that $0 \leq i \leq n-2$, $\Ext_S^{h+j-1}(I^{j}S_k,S) = 0$ for $0 \leq k \leq i+1$ and $k - h + 2 \leq j \leq i-h+3$. Note if $k = i+1 = n-1$, and $j = i-h+3 = n-h+1$, then we are done.

  To prove this, we use induction on $k$. First suppose $k = 0$. If $j \leq 0$, then $h \geq 2$ and hence $I^jS_0 = S$, so we are done.

  If $j \geq 1$, then $I^jS_0 = I^j$, as $I\cap J_k =a_k$ for $k <n$ and $a_0 = 0$.  As $\Ext_{S}^{h+j}(S/I^j,S)=0$ for $1 \leq j \leq n-h+1$, $\Ext_{S}^{h+j-1}(I^j,S)=0$ for $1 \leq j \leq i-h+3 \leq n-h+1$ and we are done.

  So, now suppose $1 \leq k \leq i+1$. We first show that the following sequence is exact: \begin{equation*}
    C_{jk}: 0 \rightarrow I^{j-1}S_{k-1} \xrightarrow{\cdot a_k} I^jS_{k-1} \rightarrow I^jS_{k}\rightarrow 0.
\end{equation*}
  If $j\leq 1$, then $k \leq h-1$, so $J_k = \mathfrak{a}_k$ and $J_{k-1}= \mathfrak{a}_{k-1}$, are complete intersections and $a_k$ is a regular element on $S_{k-1}$, so we are done. For $j \geq 2$, exactness is given by \cite[Lemma 2.7]{U} with $s = n-1$ and $r = n-h$, and $J_{n-1}$ as the geometric $n-1$-residual intersection.

  Applying $\Hom_S(-,S)$ to the exact sequence $C_{j,k}$ we get
  \begin{equation*}
    \Ext^{h+j-2}(I^{j-1}S_{k-1},S) \rightarrow \Ext^{h+j-1}(I^{j}S_{k},S) \rightarrow \Ext^{h+j-1}(I^{j}S_{k-1},S) \rightarrow \dots
  \end{equation*}

  By induction we have that $\Ext^{h+j-2}(I^{j-1}S_{k-1},S) = \Ext^{h+j-1}(I^{j}S_{k-1},S) = 0$, thus we obtain $\Ext^{h+j-1}(I^{j}S_{k},S) = 0$ and we are done.

  Since $I$ satisfies $AN_{n-1}$, $J_{n-1}$ is unmixed of height $n-1$ \cite[Proposition 1.7]{U}. Note that $\height(J_{n-1}+I) \geq n$, as $J_{n-1}$ is geometric. Thus $I$ is not contained in any associated prime of $J_{n-1}$. So, localizing at any associated prime $p$ of $A$, which is the image of an associated prime of $J_{n-1}$ in $A$, we get that $(IA)_p = A_p$. Thus, locally, at every associated prime of $A$, the map $IA\otimes_A\omega_A \rightarrow I\omega_A$ is an isomorphism. Thus the kernel of $IA\otimes_A\omega_A \rightarrow I\omega_A$ is an $A$-torsion module. So,
  \begin{equation*}
    \Hom_A(I\omega_A,\omega_A) \cong \Hom_A(IA\otimes_A\omega_A,\omega_A).
  \end{equation*}

  Now we prove that \begin{eqnarray}\label{eq2}
    \Hom_{A}(I^{n-h+1}A,\omega_A) &\cong& \mathfrak{a}A : IA.
  \end{eqnarray}

  Note, by \cite[Theorem 2.9]{U} $\omega_A \cong I^{n-h}A$. So we have
  \begin{eqnarray*}
    \Hom_{A}(I^{n-h+1}A,\omega_A) &\cong& \Hom_{A}(I\omega_{A}, \omega_{A}) \\
    &\cong& \Hom_{A}(IA\otimes_A\omega_{A}, \omega_{A})\\
    &\cong& \Hom_{A}(IA, \Hom_{A}(\omega_{A},\omega_{A}))\\
    &\cong& \Hom_{A}(IA, A)\\
    &\cong& \mathfrak{a}A : IA.
  \end{eqnarray*} The fourth isomorphism follows from the fact that $A$ is $S_2$, and the fifth holds because $\mathfrak{a}A$ is generated by a single non-zero-divisor element on $A$.

  Note that $\overline{A} \cong \overline{S}/\overline{J_{n-1}}$. Induction gives us that $\overline{J_{n-1}} \cong \overline{\mathfrak{a}_{n-1}}:\overline{I}.$ Applying \cite[Theorem 2.9]{U} we have that $\omega_{\overline{A}} \cong \overline{I^{n-h}}\overline{A}.$ By the same reasoning as above,
  \begin{equation*}
    \Hom_{\overline{A}}(\overline{I}\omega_{\overline{A}},\omega_{\overline{A}}) \cong \Hom_{\overline{A}}(\overline{I}\overline{A}\otimes_{\overline{A}}\omega_{\overline{A}},\omega_{\overline{A}}).
  \end{equation*}

  As $A$ is Cohen-Macaulay and $l$ is regular on $A$, we have that $\omega_{A}/l\omega_A \cong \omega_{\overline{A}} \cong \Hom_{\overline{A}}(\overline{A}, \omega_{\overline{A}})$.

  There exists a natural surjection $I^{n-h+1}A/lI^{n-h+1}A \rightarrow \overline{I^{n-h+1}}\overline{A}$. Note that $(IA+lA)/lA = \overline{I}\overline{A}$ and that $\height(\overline{I}\overline{A}) \geq 1$ because $\overline{J_{n-1}}$ is a geometric residual intersection of $\overline{I}$. Since $l$ is regular on $A$, we have that $\height(IA+lA) \geq 2$, which implies that locally at every minimal prime of $lA$ and for every integer $j$, $I^jA \cap lA = lI^jA$. Note because $\overline{A}$ is Cohen-Macaulay, the associated primes of $\overline{A}$ are all minimal, thus the kernel of $I^{n-h+1}A/lI^{n-h+1}A \rightarrow \overline{I^{n-h+1}}\overline{A}$ is an $\overline{A}$-torsion module. So,
  \begin{eqnarray}\label{eq3}
    \Hom_{\overline{A}}(I^{n-h+1}A/lI^{n-h+1}A,\omega_{\overline{A}})
    &\cong& \Hom_{\overline{A}}(\overline{I^{n-h+1}}\overline{A},\omega_{\overline{A}}).
  \end{eqnarray}\\

  We now prove that \begin{eqnarray}\label{eq44}
    \Hom_A(I^{n-h+1}A,\omega_A/l\omega_A)
    &\cong& \overline{\mathfrak{a}}\overline{A} : \overline{I}\overline{A}.
  \end{eqnarray}

  We have that
  \begin{eqnarray*}
    \Hom_A(I^{n-h+1}A,\omega_A/l\omega_A) &\cong& \Hom_A(I^{n-h+1}A,\Hom_{\overline{A}}(\overline{A}, \omega_{\overline{A}}))\\
    &\cong&   \Hom_{\overline{A}}(I^{n-h+1}A/lI^{n-h+1}A,\omega_{\overline{A}})\\
    &\cong& \Hom_{\overline{A}}(\overline{I^{n-h+1}}\overline{A},\omega_{\overline{A}})\\
    &\cong& \overline{\mathfrak{a}}\overline{A} : \overline{I}\overline{A}.
  \end{eqnarray*}
  The third isomorphism follows from (\ref{eq3}). The forth isomorphism follows from the same argument as the isomorphism $\Hom_{A}(I^{n-h+1}A,\omega_A) \cong \mathfrak{a}A : IA$.

  Applying (\ref{eq2}) and (\ref{eq44}) to exact sequence (\ref{homEq}) we have the exact sequence
  \begin{equation*}
    0 \rightarrow \mathfrak{a}A : IA \xrightarrow{l} \mathfrak{a}A : IA \xrightarrow{\pi}  \overline{\mathfrak{a}}\overline{A} : \overline{I}\overline{A} \rightarrow 0.
  \end{equation*}
  Thus $\overline{J} = \overline{\mathfrak{a}}:\overline{I}$.

  Now it remains to show that $l$ is regular on $S/J$. By \cite[Proposition 1.7]{U} $J$ is unmixed of height $n$. Note that $\pi$ induces a surjective map $S/J \rightarrow \overline{S}/\overline{J}$. As the height of $\overline{J}$ is at least $n$ and since all associated primes of $J$ have height $n$, $l$ is none of these associated primes, therefore $l$ is regular on $S/J$. \end{proof}

Now that we have established this general setting, we shall show that it can be applied to the setting of our main result. For $n > 5$, the main tool we will need for that is a representation theoretic result by Raicu, Weyman, and Witt \cite[Theorem 4.3]{RWW} to establish the vanishing of $\Ext$ required by the above theorem (Corollary \ref{geq5}). Note that the use of this result is why we require characteristic 0 in this chapter. However, in the case where $n =4$, this result alone is not enough. We cannot actually apply our theorem to the case of $n=4$, however we can modify the proof and apply the work of \cite{ABW} (Lemma \ref{eq4}).

Let $\mathcal{A}$, $\mathcal{B}$ be Gorenstein rings containing a field of characteristic zero. For $n \geq 4$, let $R' = \mathcal{A}[x_1,\dots,x_n,y_1,\dots,y_n]$ and $S' = \mathcal{B}[x_1,\dots,x_n,y_1,\dots,y_n]$. Select prime ideals $p$ of $R'$ and $q$ of $S'$ and let $R = R'_p$ and $S = S'_q$.

Let $M = \begin{pmatrix}
  x_1 \dots x_n \\
  y_1 \dots y_n
\end{pmatrix}$. Let $I = I_2(M) \subsetneq S$.
Define $h = \height(I) = n-1$. Let $\mathfrak{a} \subset I$ be an ideal such that $J = \mathfrak{a}:I$ is a $n$-residual intersection.

\begin{cor}\label{geq5}
  Suppose $n \geq 5$. If there exists a surjective local map $\pi: S \rightarrow R$ of algebras over $\mathbb{Z}[x_1,\dots,x_n, y_1,\dots,y_n]$ such that the kernal of $\pi$ is generated by a regular sequence $\underline{l}$, and if $\pi(\mathfrak{a}):\pi(I)$ is an $n$-residual intersection then:
  \begin{enumerate}[label=$($\alph*$)$]
    \item $\pi(J) = \pi(\mathfrak{a}):\pi(I)$
    \item $\underline{l}$ is a regular sequence in $S/J$.
  \end{enumerate}
\end{cor}

\begin{proof}

  By \cite[Theorem 4.3]{RWW} $\Ext_S^{h+j}(S/I^j,S) = 0$ for $1 \leq j \leq n-3$, so we are done by Theorem \ref{general specialization}.
\end{proof}

  \begin{lem}\label{eq4}
    Suppose $3 \leq n \leq 4$. If there exists a surjective local map $\pi: S \rightarrow R$ of algebras over $\mathbb{Z}[x_1,\dots,x_4, y_1,\dots,y_4]$ such that the kernel of $\pi$ is generated by a regular sequence $\underline{l}= l_1,\dots,l_m$ where $l_j \notin (x_1,\dots,x_4,y_1,\dots,y_4)$ for all $1 \leq j \leq m$, and if $\pi(\mathfrak{a}):\pi(I)$ is an $n$-residual intersection then:
    \begin{enumerate}[label=$($\alph*$)$]
      \item $\pi(J) = \pi(\mathfrak{a}):\pi(I)$
      \item $\underline{l}$ is a regular sequence in $S/J$.
    \end{enumerate}
  \end{lem}

  \begin{proof}

  Without loss of generality, we may assume $S$ and $R$ have infinite residue fields by attaching indeterminates onto $S$ and $R$, then localizing. Note that $I$ is a complete intersection on the punctured spectrum of $S$, and thus $I$ is $G_{8}$. Also note that $I$ is Cohen-Macaulay \cite[Corollary 2.8]{BV}.

  Let $\underline{l}_k = l_1,\dots,l_k$, and define $\pi_k$ to be the natural surjections $S \rightarrow S/(\underline{l}_k)$. By \cite[Theorem 22.6]{M2} $S/(\underline{l}_k)$ is flat over $Z$. By \cite[Lemma 4.1]{HU}, we have that $\height(\pi_k(\mathfrak{a}):\pi_k(I))\geq n$. So, we may reduce to the case where $\underline{l}$ is a single element. Let said element be $l$. Thus, we can reduce to $R = S/(l)$.

  Suppose that $n = 3$. By assumption we have that $\mu(\mathfrak{a}) = 3$. As $\pi(\mathfrak{a}):\pi(I)$ is an $3$-residual intersection, the height of $\height(\pi(\mathfrak{a})) = 3$. So, we have the inequality $3 \geq \height(\mathfrak{a}) \geq \height(\pi(\mathfrak{a})) = 3$. This implies that $\mathfrak{a}$ is a complete intersection ideal, and thus unmixed. Since $\height(\mathfrak{a}) = \height(\pi(\mathfrak{a}))$, $l$ is not contained in any minimal prime of $\mathfrak{a}$, and thus is not contained in any associated prime of $\mathfrak{a}$. Thus $l$ is regular on $S/\mathfrak{a}$, and (1) and (2) follow from \cite[Proposition 2.8]{HU-old}. 

  Now suppose $n = 4 $. Let $\overline{\cdot}$ represent images mod $l$. Note that, as $\overline{I}$ is $G_{8}$, by \cite[Corollary 1.6a]{U} we can select general elements $\underline{a} = a_1,\dots,a_4$  that are generators of $\mathfrak{a}$ such that for $\mathfrak{a}_k = (a_1,\dots,a_k)$, $\overline{\mathfrak{a}_k}:\overline{I}$ is a geometric $k$-residual intersection for all $k <4$. For notational purposes write $J_k = \mathfrak{a}_k : I$. Note, $\overline{J_3} = \overline{\mathfrak{a}_3}:\overline{I}$. As for all $k<4$, $\height(J_{k}) \geq \height(\overline{J_k})$ and  $\height(I + J_{k}) \geq \height(\overline{I+J_k})$, $J_k$ is a geometric $k$-residual intersection. 

  Let $S_k = S/J_k$ and $A = S_{3}$. Note that $A$ is Cohen-Macaulay \cite{PS}. Let $\omega_A$ be the canonical module of $A$. Note that if an ideal is contained in $S$, $\overline{\cdot}$ represents images in $R$, and if an ideal is contained in $A$, $\overline{\cdot}$ represents images in $\overline{A} = A/lA$.

  Note, that by induction $l$ is regular on $A$ and thus we have the exact sequence
  \begin{equation*}
    0 \rightarrow \omega_A \xrightarrow{\cdot l} \omega_A \rightarrow \omega_A/l\omega_A \rightarrow 0.
  \end{equation*}

  We will first show that, after applying the functor $\Hom_A(I^{2}A,-)$, we have the short exact sequence
  \begin{equation}\label{homEq-4}
    0 \rightarrow \Hom_A(I^{2}A,\omega_A) \xrightarrow{\cdot l} \Hom_A(I^{2}A,\omega_A) \rightarrow \Hom_A(I^{2}A,\omega_A/l\omega_A) \rightarrow 0.
  \end{equation}

  Note that after applying $\Hom_A(I^2A,-)$, we have the long exact sequence

  \begin{equation*}
    \dots \rightarrow \Hom_A(I^{2}A,\omega_A/l\omega_A) \rightarrow \Ext_A^1(I^{2}A,\omega_A)\xrightarrow{\phi}  \Ext_A^1(I^{2}A,\omega_A) \rightarrow \dots
  \end{equation*}

  and so it enough to show that $\phi$ is injective. Note that $\phi$ is just multiplication by $l$.

  As $A$ is Cohen-Macaulay \cite{PS}, we have that $\Ext_A^1(I^{2}A,\omega_A) \cong \Ext_S^{4}(I^{2}A,\omega_S)$ and since $S$ is Gorenstein, $\omega_S \cong S$.

  Note, since $I$ is Cohen-Macaulay it satisfies $AN_3^-$ \cite{PS}. So, from \cite[Lemma 2.7]{U}, setting $r=1$, we have the exact sequences

  \begin{equation*}
    C_{i,j}: 0 \rightarrow I^{j-1}S_{i-1} \rightarrow I^{j}S_{i-1} \rightarrow I^{j}S_{i} \rightarrow 0
  \end{equation*}

  for $1 \leq i \leq 4$ and $j \geq 2$.

  Note that $\dim(S) - \dim(S/I) = 3$ and  $\dim(S)-\depth(S/I) = 3$. As $S$ is Gorenstein, $0 = \Ext^i_S(S/I, S) \cong \Ext_S^{i-1}(I,S)$ for all $i \neq 3$. In particular $\Ext_S^3(I,S) = \Ext_S^4(I,S) = 0$. Applying $\Hom_S(-,S)$ to $C_{1,2}$ gives us $\Ext_S^4(I^2S_1,S) \cong \Ext_S^4(I^2,S)$.

  From \cite[Theorem 4.3]{RWW} and the proof of \cite[Theorem 4.1]{CEU}, we have $\Ext_S^{3+j-1}(I^jS_k,S)=0$ whenever $k-1\leq j \leq 1$ and $0 \leq k \leq 2$. Here we will not reprove the relevant results, as they are made explicit in the proof of Theorem \ref{general specialization}.

  So we have that $\Ext_S^3(IS_1,S) = \Ext_S^3(IS_2,S) = 0$. Applying $\Hom_S(-,S)$ to $C_{2,2}$ and $C_{3,2}$ we get that $\Ext_{S}^4(I^2S_2,S)\hookrightarrow \Ext_S^4(I^2S_1,S)$ and $\Ext_{S}^4(I^2S_3,S)\hookrightarrow \Ext_S^4(I^2S_2,S)$. So $\Ext_{S}^4(I^2S_3,S)\hookrightarrow \Ext_S^4(I^2,S)$.

  By \cite{ABW} we have that the resolution of $I^2$ is linear, that the last Betti number is $1$ and that the projective dimension of $I^2$ is $4$. Computing $\Ext$ gives us $\Ext_S^4(I^2,S) \cong S/(x_1,\dots,x_4,y_1,\dots,y_4) \cong \mathcal{B}$. Since $l_1$ a non-zero-divisor in $S$ and not contained in the ideal $(x_1,\dots,x_4,y_1,\dots,y_4)$, it is a nonzero divisor on $S/(x_1,\dots,x_4,y_1,\dots,y_4)$. Thus $\phi$ is injective and we are done.

  Since $\Ass(R/J_3) \subseteq \Ass(R/\mathfrak{a}_3)$ and $\mathfrak{a_3}$ is a complete intersection, $J_{3}$ is unmixed of height $3$. As $J_{3}$ and $I$ are geometrically linked, $I$ is not contained in any associated prime of $J_{3}$. So, localizing at any associated prime, $p$, of $A$, which is the image of an associated prime of $J_{3}$ in $A$, we get that $(IA)_p = A_p$. Thus, locally, at every associated prime of $A$, the map $IA\otimes_A\omega_A \rightarrow I\omega_A$ is an isomorphism. Thus the kernel of $IA\otimes_A\omega_A \rightarrow I\omega_A$ is an $A$-torsion module. So,
  \begin{equation*}
    \Hom_A(I\omega_A,\omega_A) \cong \Hom_A(IA\otimes_A\omega_A,\omega_A).
  \end{equation*}

  Now we prove that: \begin{eqnarray}\label{eq2-4}
    \Hom_{A}(I^{2}A,\omega_A) &\cong& \mathfrak{a}A : IA.
  \end{eqnarray}

  Note $\omega_A \cong IA$ \cite{PS}. So we have
  \begin{eqnarray*}
    \Hom_{A}(I^{2}A,\omega_A) &\cong& \Hom_{A}(I\omega_{A}, \omega_{A}) \\
    &\cong& \Hom_{A}(IA\otimes_A\omega_{A}, \omega_{A})\\
    &\cong& \Hom_{A}(IA, \Hom_{A}(\omega_{A},\omega_{A}))\\
    &\cong& \Hom_{A}(IA, A)\\
    &\cong& \mathfrak{a}A : IA.
  \end{eqnarray*} The fourth isomorphism follows from the fact that $A$ is $S_2$ and the fifth holds because $\mathfrak{a}A$ is generated by a single nonzero divisor element on $A$.

  Note that $\overline{A} \cong \overline{S}/\overline{J_{3}}$. We have that $\overline{J_{3}} \cong \overline{\mathfrak{a}_{3}}:\overline{I}$, and $\omega_{\overline{A}} \cong \overline{I}\overline{A}$ \cite{PS}. By the same reasoning as above,
  \begin{equation*}
    \Hom_{\overline{A}}(\overline{I}\omega_{\overline{A}},\omega_{\overline{A}}) \cong \Hom_{\overline{A}}(\overline{I}\overline{A}\otimes_{\overline{A}}\omega_{\overline{A}},\omega_{\overline{A}}).
  \end{equation*}

  As $A$ is Cohen-Macaulay and $l$ is regular on $A$, we have that $\omega_{A}/l\omega_A \cong \omega_{\overline{A}} \cong \Hom_{\overline{A}}(\overline{A}, \omega_{\overline{A}})$.

  There exists a natural surjection $I^{2}A/lI^{2}A \rightarrow \overline{I^{2}}\overline{A}$. Note that $(IA+lA)/lA = \overline{I}\overline{A}$ and that $\height(\overline{I}\overline{A}) \geq 1$ because $\overline{J_{3}}$ and $\overline{I}$ are geometrically linked. Since $l$ is regular on $A$, we have that $\height(IA+lA) \geq 2$, which implies that locally at every minimal prime of $lA$ and for every integer $j$, $I^jA \cap lA = lI^jA$. Note because $\overline{A}$ is Cohen-Macaulay, the associated pimes of $\overline{A}$ are all minimal, thus the kernal of $I^{2}A/lI^{2}A \rightarrow \overline{I^{2}}\overline{A}$ is an $\overline{A}$-torsion module. So,
  \begin{eqnarray}\label{eq3-4}
    \Hom_{\overline{A}}(I^{2}A/lI^{2}A,\omega_{\overline{A}})
    &\cong& \Hom_{\overline{A}}(\overline{I^{2}}\overline{A},\omega_{\overline{A}}).
  \end{eqnarray}\\

  We now prove that: \begin{eqnarray}\label{eq44-4}
    \Hom_A(I^{2}A,\omega_A/l\omega_A)
    &\cong& \overline{\mathfrak{a}}\overline{A} : \overline{I}\overline{A}.
  \end{eqnarray}

  We have that
  \begin{eqnarray*}
    \Hom_A(I^{2}A,\omega_A/l\omega_A) &\cong& \Hom_A(I^{2}A,\Hom_{\overline{A}}(\overline{A}, \omega_{\overline{A}}))\\
    &\cong&   \Hom_{\overline{A}}(I^{2}A/lI^{2}A,\omega_{\overline{A}})\\
    &\cong& \Hom_{\overline{A}}(\overline{I^{2}}\overline{A},\omega_{\overline{A}})\\
    &\cong& \overline{\mathfrak{a}}\overline{A} : \overline{I}\overline{A}.
  \end{eqnarray*}
  The third isomorphism follows from (\ref{eq3-4}). The forth isomorphism follows from the same argument as the isomorphism $\Hom_{A}(I^{2}A,\omega_A) \cong \mathfrak{a}A : IA$.

  Applying (\ref{eq2-4}) and (\ref{eq44-4}) to exact sequence (\ref{homEq-4}) we have the exact sequence
  \begin{equation*}
    0 \rightarrow \mathfrak{a}A : IA \xrightarrow{l} \mathfrak{a}A : IA \xrightarrow{\pi}  \overline{\mathfrak{a}}\overline{A} : \overline{I}\overline{A} \rightarrow 0.
  \end{equation*}
  Thus $\overline{J} = \overline{\mathfrak{a}}:\overline{I}$.

  It remains to show that $l$ is regular on $S/J$. By \cite[Proposition 1.7]{U} $J$ is unmixed of height $n$. Note that $\pi$ induces a surjective map $S/J \rightarrow \overline{S}/\overline{J}$. As the height of $\overline{J}$ is at least $n$ and since all associated primes of $J$ have height $n$, $l$ is none of these associated primes, therefore $l$ is regular on $S/J$. \end{proof}

Now we prove our main result for the special case of generic residual intersections. Let $\mathbf{k}$ be a field of characteristic 0. For $n \geq 4$, let $r = \binom{n}{2}$, let $R = \mathbf{k}[x_1,\dots,x_n,y_1,\dots,y_n,z_1,\dots,z_{n}]$, and let $S = \mathbf{k}[x_1,\dots,x_n,y_1,\dots,y_n][z_{i,j} \;\vert \; 1 \leq i \leq r, 1 \leq j \leq n]$.

Let $M = \begin{pmatrix}
  x_1 \dots x_n \\
  y_1 \dots y_n
\end{pmatrix}$. Let $I = I_2(M) = (g_1,\dots,g_{r}) \subseteq S$ be an ideal where $g_1,\dots,g_{r}$ are binomials, $g_1 = \Delta_{2,1}$, $g_n = \Delta_{n,n-1}$  and, for $2 \leq i \leq n-1$, $g_i = \Delta_{i+1,i-1}$.
Define $h = \height(I) = n-1$. Let $B = (z_{i,j})$ be a generic $r \times n$ matrix, let $[a_1,\dots,a_n] = [g_1,\dots,g_r]B$, and let $\mathfrak{a} = (a_1,\dots,a_n)$. The ideal $J = \mathfrak{a}:I$ is a generic residual intersection as defined by Huneke and Ulrich in \cite[Definition 3.1]{HU}, and, as proved in the same paper, $J$ is a geometric $n$-residual intersection \cite[Theorem 3.3]{HU}.

\begin{lem}\label{MainLem}
  The residual intersection $J$ is equal to $\sum_{i=1}^{n} (a_1,\dots,\widehat{a}_i,\dots,a_n):I$.
\end{lem}

\begin{proof}
Since all the ideals we are dealing with are homogeneous, without loss of generality, we may localize $R$ and $S$ at their homogeneous maximal ideals.

Define $\pi: S\rightarrow R$ as the map that sends $z_{i,i} \mapsto z_i$ for all $i$ and $z_{i,j} \mapsto 0$ for all $i \neq j$. Let $\mathfrak{a}_i = (a_1, \dots, \widehat{a_i}, \dots, a_n)$, $J_i = \mathfrak{a}_i:I$ and $J' = \sum_{i=1}^n J_i$. By Corollary \ref{geq5} and Lemma \ref{eq4} we have that $\pi(J) = \pi(\mathfrak{a}):\pi(I)$ and that $\pi(J') = \sum_{i=1}^n\pi(\mathfrak{a}_i):\pi(I)$. Note that these images are the sum of links and residual intersection from Section \ref{Sum of Links}, so by Theorem \ref{sum of links} we have that $\pi(J) = \pi(J')$.

Note that $\ker(\pi)$ is generated by regular sequence in $S$ and by Corollary \ref{geq5} and Lemma \ref{eq4} it is generated by a regular sequence in $S/J$. Note $J' \subseteq J$ and, since $\pi(J) = \pi(J')$, $J \subseteq J' + \ker(\pi)$.

So we have: $J = (J' + \ker(\pi))\cap J = J' +  \ker(\pi)\cap J = J' +  \ker(\pi)J$, where the last equality holds because $\ker(\pi)$ is generated by a regular sequence in $S/J$. So, by Nakayama's lemma, $J = J'$. \end{proof}

Now, we are finally ready to prove our main result. Let $\mathbf{k}$ be a field of characteristic 0. For an integer $n \geq 4$, let $R$ be $\mathbf{k}[x_1,\dots,x_n,y_1,\dots,y_n]$ localized at its homogeneous maximal ideal, let $M = \begin{pmatrix}
  x_1 \dots x_n \\
  y_1 \dots y_n
\end{pmatrix}$ be a $2\times n$ generic matrix, and let $I \subsetneq R$ be the ideal generated by the $2\times 2$ minors of $M$.

\mainResult

\begin{proof}
  This follows from Corollary \ref{geq5}, Lemma \ref{eq4} and Lemma \ref{MainLem}. \end{proof}

\bibliographystyle{amsplain}
\bibliography{Bibliography}

\end{document}